\documentclass[a4paper,11pt, reqno]{amsart}
\usepackage{amssymb,amsmath}
\usepackage{cite}

\usepackage{amsmath,amscd,amssymb}
\usepackage{latexsym}
\usepackage[colorlinks,citecolor=blue,pagebackref,hypertexnames=false]{hyperref}

\numberwithin{equation}{section}
\allowdisplaybreaks[2]
\theoremstyle{plain}
\newtheorem{theorem}{Theorem}[section]

\newtheorem{lemma}[theorem]{Lemma}
\newtheorem{corollary}[theorem]{Corollary}
\newtheorem{proposition}[theorem]{Proposition}

\theoremstyle{definition}

\newtheorem{problem}[theorem]{Problem}

\theoremstyle{remark}
\newtheorem{remark}[theorem]{Remark}

\newtheorem{case[theorem]}{Case}

\def\norm#1.#2.{\lVert#1\rVert_{#2}}

\title[Orthonormal  Strichartz inequalities for  the $\Delta_{k, a}$ and  Dunkl operator]{Orthonormal  Strichartz inequalities for the  $(k, a)$-generalized Laguerre operator and  Dunkl operator
}

\author{Shyam Swarup Mondal} 
\author{Manli Song}

\address{ \endgraf Department of Mathematics, Indian Institute of Technology Delhi, Delhi 11016, India} 
\email{mondalshyam055@gmail.com}

\address{School of Mathematics and Statistics, Northwestern Polytechnical University, Xi’an, Shaanxi 710129, China}
\email{mlsong@nwpu.edu.cn} 


\date{\today}
\begin{document}
	
	\maketitle

	\allowdisplaybreaks

	\begin{abstract} Let $\Delta_{k,a}$ and   $\Delta_k $  be the $(k,a)$-generalized Laguerre operator  and   the Dunkl Laplacian operator  on $\mathbb{R}^n$, respectively.   The aim of this article is twofold.  First,   we prove a restriction theorem for  the Fourier-$\Delta_{k,a}$ transform. Next,  as an application of the restriction problem, we  establish   Strichartz    estimates	  for orthonormal families of initial data    for the Schr\"odinger propagator  $e^{-i t \Delta_{k, a}} $ associated with     the operator  $ \Delta_{k, a}$.   	 	Further, using the classical Strichartz estimates for the free Schr\"odinger propagator $e^{-i t \Delta_{k, a}} $  for  orthonormal systems of initial data and the kernel relation  between the semigroups $e^{-i t \Delta_{k, a}}$ and $e^{i \frac{t}{a}\|x\|^{2-a} \Delta_{k}},$ we prove   Strichartz estimates for orthonormal systems of initial data    associated with  the Dunkl operator $ \Delta_k $ on $\mathbb{R}^n$.  Finally,  we present some applications  to our   aforementioned results.

	\end{abstract}
	\tableofcontents 	
	
	\section{Introduction}
	One of the classical problems in harmonic analysis is the so-called restriction theorem. Historically, the restriction problem originated from studying the boundedness of Fourier transform of an $L^p$-function in the Euclidean space $\mathbb{R}^n$ for some $n  \geq2$. Later, it was realized that it also arose naturally in other contexts, such as nonlinear PDE and the study of eigenfunctions of the Laplacian. However, we recall the most simplest case. For a Schwartz class function  $f$ on $\mathbb{R}^n$, the  Fourier transform and the inverse Fourier transform of $f$ are defined as 
	$$
	\mathcal{F}(f)(\xi)=(2 \pi)^{-n} \int_{\mathbb{R}^n} f(x) e^{-i x\cdot \xi } d x, \quad   \xi \in \mathbb{R}^{n}  ,
	$$ and $$\mathcal{F}^{-1}(f)(x)=\int_{\mathbb{R}^n} f(\xi) e^{i x\cdot \xi} d \xi, \quad   x \in \mathbb{R}^{n} ,$$ respectively. Then the classical restriction problem is stated as follows.
	
	\begin{problem}\label{1.1}
		Given a surface $S$ embedded in $\mathbb{R}^{n}$ with $n \geq 2$, for which exponents $1 \leq p \leq 2,~1 \leq q \leq \infty$, the Fourier transform of a function $f \in L^{p}\left(\mathbb{R}^{n}\right)$ belongs to $L^{q}(S)$, where $S$ is endowed with its $(n-1)$-dimensional Lebesgue measure $d \sigma$? i.e.,
		\begin{equation*}
			\|f\|_{L^{q}(S)}\leq C\|f\|_{L^{p}\left(\mathbb{R}^{n}\right)}?
		\end{equation*}
	\end{problem}

	The restriction theorem for the Fourier transform plays an important role in harmonic analysis as well as in the theory of partial differential equations. To understand the general oscillatory integral operators, it plays a crutial role.    Moreover, it  is also connected to many other important problems in mathematical analysis, geometric measure theory,  combinatorics, harmonic analysis, number theory, including several  conjectures such as the Bochner-Riesz conjecture, Kakeya conjecture, the estimation of solutions to the wave, Schr\"odinger, and Helmholtz equations, and the local smoothing conjecture for PDE's \cite{tao}.
	
	The case in which  the restriction problem often considered in the literature is $q=2$. There are only two types of surfaces for which this problem has been completely settled, namely, smooth compact surfaces with non-zero Gauss curvature and quadratic surfaces. In this direction, an affirmative answer to Fourier restriction problem for compact surfaces with non-zero Gaussian curvature can be seen in the works of Stein and Tomas  \cite{Stein,T1,T2}. On the other hand,   for  quadratic  surfaces, Strichartz \cite{st}  gave a complete charecterization  by spliting the surface    into three categories such as  parabloid-like, cone-like, or sphere-like.   We refer to the excellent survey of    Tao \cite{tao} for a   detailed study on the historical background  of the restriction problem.
	
	The main aim of this  paper is  to investigate the validity of Problem \ref{1.1} for the  Fourier-$\Delta_{k, a}$ transform and obtain Strichartz type estimates for a system of orthonormal functions for the $(k,a)$-generalized Laguerre operator $\Delta_{k, a}$ and Dunkl Laplacian operator  $\Delta_{k}$ on $\mathbb{R}^n$.

	\subsection{Restriction problem for the    Fourier-$\Delta_{k, a}$ transform}   We refer to Subsection \ref{dunklop} for the notational convention used in this subsection. 
	
	Let $f\in L^1_{k,a}(\mathbb{R}^n)$. Suppose  $a>0 $ and $ k\in\mathcal{K}^+$ such that  $a+2\gamma+n-2>0$. Then the   $(k, a)$-generalized Laguerre  transform of $f$ is defined by 
	$$\hat{f}(\ell, m , j ):=\int_{\mathbb{R}^n}f(x)\Phi_{\ell,m,j}^{(a)}(x)\,v_{k, a}(x) dx, \quad  \ell,m\in\mathbb{N}, j\in J_m,$$
	where $\Phi_{\ell,m,j}^{(a)}, v_{k, a}$ and $J_m$ are all defined in Subsection \ref{dunklop}. Let $\mathcal{A}:=\{(\ell, m , j ) : \ell, m\in \mathbb{N} , j\in J_m  \}$ be a discrete set with respect to the counting measure. If $f\in L^2_{k,a}(\mathbb{R}^n)$, then $\hat{f}(\ell, m , j ) \in \ell^2(\mathcal{A})$ and the Plancherel formula is of the form  $$\|f\|_{ L^2_{k,a}(\mathbb{R}^n)}^2= \sum_{(\ell, m , j ) \in \mathcal{A}}|\hat{f}(\ell, m , j ) |^2.$$
	The inverse  $(k, a)$-generalized Laguerre transform is given by $$f(x):= \sum_{(\ell, m , j) \in   \mathcal{A}} \hat{f}(\ell, m , j )\Phi_{\ell,m,j}^{(a)}(x) ,$$
	for all $f\in \mathcal{S}(\mathbb{R}^n, v_{k, a}(x)dx).$
	
	Strichartz \cite{st} first observed that restriction theorems for some quadradic surfaces are closely related to space-time decay estimates (also called Strichartz estimates) for some envolution equations.  If we consider the free Schr\"odinger equation associated with the $(k, a)$-generalized Laguerre operator   (see \eqref{100}), then the solution is given by  $u(t,x)=e^{-i t\Delta_{k, a}} f(x)$ and the $L_{k,a}^p(\mathbb{R}^n)$ norm is $\pi$-periodic in $t$ variable. 	Thus by Strichartz \cite{st}, the Strichartz estimate for the  Sch\"odinger propagator $e^{-i t\Delta_{k, a}} $ is also reduced to the restriction theorem on $ \mathbb{T}\times \mathbb{R}^{n},$ where $\mathbb{T}=(-\pi, \pi)$. Thus, in order to obtain the Strichartz inequality for the solution $u(t,x)=e^{-i t\Delta_{k, a}} f(x)$,  we need to introduce the Fourier-$\Delta_{k,a}$ transform on $\mathbb{T}\times \mathbb{R}^{n}$.  
	
	For any $F \in L^{1}\left(\mathbb{T}, L^1_{k,a}(\mathbb{R}^n)\right)$, the Fourier-$\Delta_{k,a}$ transform of $F$ is given by
	$$
	\hat{F}(\nu, \ell, m , j ):=\frac{1}{2 \pi} \int_{\mathbb{R}^{n}} \int_{\mathbb{T}} F(t,x) \Phi_{\ell,m,j}^{(a)}(x) e^{i t \nu}\, d t \,v_{k, a}(x) dx, \quad \forall \nu \in \mathbb{Z}, (\ell, m , j ) \in\mathcal{A}.
	$$
	If $F \in L^{2}\left(\mathbb{T}, L^2_{k,a}(\mathbb{R}^n)\right)$, then $\{ \hat{F}(\nu, \ell, m , j)\}_{(\nu, \ell, m , j) \in \mathbb{Z}\times \mathcal{A}} \in \ell^{2}\left(\mathbb{Z}\times \mathcal{A}\right)$, and the Plancherel formula is of the form
	$$
	\|F\|_{L^{2}\left(\mathbb{T}, L^2_{k,a}(\mathbb{R}^n)\right)}=\sqrt{2 \pi}\|\{\hat{F}(\nu,\ell, m , j)\}\|_{\ell^{2}\left(\mathbb{Z}\times \mathcal{A}\right)} .
	$$
	Moreover, the inverse Fourier-$\Delta_{k,a}$ transform is defined as  
	$$F(t,x):= \sum_{(\nu,\ell, m , j) \in \mathbb{Z}\times \mathcal{A}} \hat{F}(\nu,\ell, m , j )\Phi_{\ell,m,j}^{(a)}(x)e^{-i t \nu} , \quad (t, x)\in \mathbb{T} \times \mathbb{R}^n.$$
	Given a discrete surface $S$ in $\mathbb{Z}\times \mathcal{A}$, we define the restriction operator $\mathcal{R}_S$ as $$\mathcal{R}_SF:=\{ \hat{F}(\nu,\ell, m , j) \}_{(\nu,\ell, m , j)\in S}$$ 
	and the operator dual to $\mathcal{R}_S$ (called the extension Fourier-$\Delta_{k,a}$ operator) as
	\begin{align}\label{extension}
		\mathcal{E}_S(\{\hat{F}(\nu,\ell, m , j)\})(t,x):= \sum_{(\nu,\ell, m , j) \in S} \hat{F}(\nu,\ell, m , j)\Phi_{\ell,m,j}^{(a)}(x)e^{-i t \nu} .
	\end{align}
	Now we consider the following restriction problem:
	
	{\bf Problem 1:} For which exponents of $p, q$ ,$1\leq p,q\leq 2,$ the sequence of the Fourier-$\Delta_{k,a}$ transform   of a function $F\in L^{q}\left(\mathbb{T}, L^p_{k,a}(\mathbb{R}^n)\right)$ belongs to $\ell^2(S)$? i.e., 
	$$\|\mathcal{R}_SF\|_{\ell^2(S)}\leq C \|F\|_{L^{q}\left(\mathbb{T}, L^p_{k,a}(\mathbb{R}^n)\right) }?$$

	\noindent	Let $p'$ and $q'$ are  the congugate exponents of $p$ and $q$, respectively,  then    $\frac{1}{p}+\frac{1}{p^{\prime}}=1$ and $\frac{1}{q}+\frac{1}{q^{\prime}}=1$. Thus by the duality argument, Problem 1 is  equivalent to the  boundedness of the extension operator $\mathcal{E}_{S}$ as  follows:
	
	{\bf Problem 2:} For which exponents of $p, q$ ,$1\leq p,q\leq 2,$ the extension Fourier-$\Delta_{k,a}$ operator is bounded from  $\ell^2(S)$ to  $ L^{q'}\left(\mathbb{T}, L^{p'}_{k,a}(\mathbb{R}^n)\right)$?  i.e., 
	$$ \|	\mathcal{E}_S(\{\hat{F}(\nu,\ell, m , j)\})\|_{L^{q'}\left(\mathbb{T}, L^{p'}_{k,a}(\mathbb{R}^n)\right) }\leq C \|	 (\{\hat{F}(\nu,\ell, m , j)\})\|_{\ell^2(S)}?$$

	\noindent	Since $\mathcal{E}_S$ is bounded from $\ell^2(S)$ to $L^{q'}\left(\mathbb{T}, L^{p'}_{k,a}(\mathbb{R}^n)\right)$ if and only if $\mathcal{T}_S:=\mathcal{E}_S(\mathcal{E}_S)^*$ is bounded from $L^{q}\left(\mathbb{T}, L^{p}_{k,a}(\mathbb{R}^n)\right)$ to $L^{q'}\left(\mathbb{T}, L^{p'}_{k,a}(\mathbb{R}^n)\right)$, then Problem 2 also can be reformed as:

	{\bf Problem 3:} For which exponents of $p, q$ ,$1\leq p,q\leq 2,$ the operator $ \mathcal{T}_{S}$ is bounded from $L^{q}\left(\mathbb{T}, L^{p}_{k,a}(\mathbb{R}^n)\right)$ to $L^{q'}\left(\mathbb{T}, L^{p'}_{k,a}(\mathbb{R}^n)\right)$? i.e.,
	$$
	\left\| \mathcal{T}_{S} F\right\|_{L^{q'}\left(\mathbb{T}, L^{p'}_{k,a}(\mathbb{R}^n)\right)} \leq C\|F\|_{L^{q}\left(\mathbb{T}, L^{p}_{k,a}(\mathbb{R}^n)\right)} ?
	$$ 
	In this paper, we find exponents of $p, q$,$1\leq p,q\leq 2,$ such  that Problem 1 has an affirmative answer when the surface $S \subset \mathbb{Z}\times \mathcal{A}$   is  discrete.

	
	\subsection{Orthonormal  Strichartz inequalities for the  $\Delta_{k, a}$ operator}
	Generalization involving the orthonormal system is strongly motivated by the theory of  many body quantum mechanics.  In quantum mechanics, a system of  $N$ independent fermions in the Euclidean space $\mathbb{R}^{n}$  can be  described by a collection of  $N$ orthonormal functions $u_{1}, \ldots, u_{N}$ in $L^{2}\left(\mathbb{R}^{n}\right)$.  It is  then  essential to obtain functional inequalities on these systems whose behavior is optimal in the finite number $N$ of such orthonormal functions.  For this particular reason, functional inequalities involving a large number of orthonormal functions are very useful in   mathematical analysis of large quantum systems.

	The idea of extending  functional inequalities involving a single function to  a system of orthonormal functions is hardly a new topic. The first initiative work  of such    generalization goes back to the famous work  established by Lieb and Thirring, known as  Lieb-Thirring inequality \cite{lieb, liebb} that generalizes   the Gagliardo-Nirenberg-Sobolev inequality, which is   one of the fundamental tools to  prove  the stability of matter  \cite{lieb}.   In 2013, 	the    classical Strichartz inequality for the Schr\"odinger propagator $e^{it\Delta}$  has  been substantially generalized  for  a system of orthonormal functions in the works of  Frank-Lewin-Lieb-Seiringer \cite{frank}.  Further,   Frank-Sabin \cite{FS} improves this result  by solving an open problem of  \cite{frank}   about the  missing interval  and proved   the full range orthonomal Strichartz inequality  by obtaining  a duality principle in terms of Schatten class and generalized the theorems of Stein-Tomas and Strichartz related to the  surfaces restrictions of Fourier transforms to systems of orthonormal functions.  The duality principle  gives an advantage   which   allows to deduce   the restriction bounds for orthonormal systems from Schatten bounds for a single function. 
	Thus the main focus is about how to get the Schatten bounds.   
	
	Stein \cite{stein} and Strichartz \cite{st} proved the boundedness of $\mathcal{T}_{S}$ by introducing an analytic family of operators $\left\{T_{z}\right\}$ in the sense of Stein defined in a strip $a \leq Re(z) \leq b$ in the complex plane such that $\mathcal{T}_{S}=T_{c}$ for some $c \in[a, b]$.  Further,  they proved that $T_{z}$ is $L^{2}-L^{2}$ bounded on the line $\operatorname{Re}(z)=b$ and $L^{1}-L^{\infty}$ bounded of on the line $\operatorname{Re}(z)=a$.   Using Stein's complex interpolation theorem \cite{stein}, they deduced the $L^{p}-L^{p^{\prime}}$ boundedness of $\mathcal{T}_{S}$ for some $p \in[1,2]$.  Note that H\"older's inequality implies that the operator  $\mathcal{T}_{S}$ is $L^{p}-L^{p^{\prime}}$ bounded if and only if the operator   $W_1\mathcal{T}_{S}W_2$ (composition of the multiplication operator associated with $W_1, \mathcal{T}_{S} $ and the multiplication operator associated with $W_2$) is  $L^{2}-L^{2}$ bounded.     Indeed, Frank and Sabin \cite{FS} proved a  more stronger Schatten bound for $W_{1} \mathcal{T}_{S} W_{2}$ which is a more general result than $L^{2}-L^{2}$ boundedness.

	Recently,   a considerable attention has been devoted  to extend   Strichartz inequality for  a system of orthonormal functions in different framework  by several researchers.  For example,  the restriction theorems  and   the Strichartz inequality for  the system of orthonormal functions      for  Hermite operator   and  special Hermite  can be found in  \cite{shyam, shyam1,  lee}. The second author with Feng   \cite{manli} established  the restriction theorem and proved orthonormal Strichartz estimate with respect to the Laguerre operator.   Moreover,  Nakamura in \cite{nakamura}  proved  the sharp orthonormal Strichartz inequality on the torus that  generalizes   Strichartz inequality on torus.   We also refer to   \cite{BEZ} for the recent developments in  the direction of   Strichartz estimates for orthonormal families of initial data and weighted oscillatory integral.  
	
	Motivated by the  recent  works and developments in the direction of orthnomal Strichartz inequality, in this paper we    establish a restriction theorem  for    Fourier-$\Delta_{k,a}$ transform and  as an application of the   restriction problem, we prove   Strichartz    inequality		for  a system of orthonormal functions   for the Schr\"odinger propagator  $e^{-i t \Delta_{k, a}} $ associated with $(k,a)$-generalized Laguerre operator $ \Delta_{k, a}.$   	Finally, 	 we aim to prove   Strichartz    inequality		for  orthonormal systems  of initial data   associated with  Dunkl operator $ \Delta_k $ on $\mathbb{R}^n$.  Here we note that,  the  results  in this paper are more general and  Strichartz    inequality		for  a system of orthonormal functions associated with   Laplacian, Hermite, and  Laguerre operator on $\mathbb{R}^n$ will be a particular case of our results.   The main results of this papers are as follows, in particular,  we consider the restriction estimate of the form:  
	\begin{theorem}[Restriction estimates for orthonormal functions-general case]	Suppose $a=1,2$ and $k$ is a non-negative multiplicity function such that
		$
		a+2 \gamma+n-2>0.
		$ Let $n\geq 1$ and $S=\{(\nu, \ell, m , j)\in \mathbb{Z}\times \mathcal{A}: \nu=2\ell+\frac{2m}{a}\}$. If $p, q, n \geqslant$ 1 such that
		$$
		1 \leqslant p<\frac{4\gamma+2n+3a-4}{  4\gamma+2n+a-4} \quad \text { and } \quad 	\frac{1}{q}+\frac{2 \gamma+n+a-2}{pa}=\frac{2 \gamma+n+a-2}{a},
		$$
		for any (possible infinity) orthonormal system $\left\{\{\hat{F_\iota}(\nu, \ell, m , j)\}_{(\nu, \ell, m , j)\in \mathbb{Z}\times \mathcal{A}}\right\}_{\iota \in{I}}$  in $\ell^{2}(S)$ and any sequence $\left\{n_{\iota }\right\}_{\iota \in I}$ in $\mathbb{C}$, we have
		\begin{align}\label{p-q range}
			\left \| \sum_{\iota \in I} n_{\iota } \left|  \mathcal{E}_{S} \left\{\hat{F}_{\iota }(\nu, \ell, m , j)\right\} \right|^{2} \right\|_{L^q((-\frac{\pi}{2},\frac{\pi}{2}), L_{k,a}^p(\mathbb{R}^n))} \leqslant C\left(\sum_{\iota }\left|n_{\iota }\right|^{\frac{2 p}{p+1}}\right)^{\frac{p+1}{2 p}},
		\end{align}
		where $C>0$ is independent of  $\left\{\{\hat{F_\iota}(\nu, \ell, m , j)\}_{(\nu, \ell, m , j)\in \mathbb{Z}\times \mathcal{A}}\right\}_{\iota \in{I}}$  and  $\left\{n_{\iota}\right\}_{\iota  \in J}$.
		
	\end{theorem} 
	As a consequences of the  restriction problem, we obtain the following Strichartz inequality for the system of orthonormal functions.
	\begin{theorem} [Strichartz inequality for orthonormal functions-general case] If $p, q, n \geqslant$ 1 such that
		$$
		1 \leqslant p<\frac{4\gamma+2n+3a-4}{  4\gamma+2n+a-4} \quad \text { and } \quad 	\frac{1}{q}+\frac{2 \gamma+n+a-2}{pa}=\frac{2 \gamma+n+a-2}{a}.
		$$
		Then for any (possible infinity) system $\left\{f_{\iota}\right\}_{\iota \in I }$ of orthonormal functions in $L_{k,a}^{2}\left(\mathbb{R}^n\right)$ and any coefficients $\left\{n_{\iota}\right\}_{\iota \in I }$ in $\mathbb{C}$, we have
		$$
		\left\|\sum_{  \iota \in I} n_{\iota }\left|e^{-i t \Delta_{k, a}}f_{\iota }\right|^{2}\right\|_{L^q ((-\frac{\pi}{2},\frac{\pi}{2}), L_{k,a}^p(\mathbb{R}^n))} \leqslant C\left(\sum_{\iota \in I}\left|n_{\iota }\right|^{\frac{2 p}{p+1}}\right)^{\frac{p+1}{2 p}},
		$$
		where $C>0$ is independent of $\left\{f_{\iota}\right\}_{\iota\in I}$ and $\left\{n_{\iota}\right\}_{\iota \in I}$.
	\end{theorem}
	Based on the above result and the kernel relation  between the semigroups $e^{-i t \Delta_{k, a}}$ and $e^{i \frac{t}{a}\|x\|^{2-a} \Delta_{k}}$, we     obtain the following Strichartz inequality associated with the Dunkl operator  for the system of orthonormal functions.
	\begin{theorem} [Strichartz inequality for orthonormal functions associated with Dunkl operator-general case] If $p, q, n \geqslant$ 1 such that
		$$
		1 \leqslant p<\frac{4\gamma+2n+3a-4}{  4\gamma+2n+a-4} \quad \text { and } \quad 	\frac{1}{q}+\frac{2 \gamma+n+a-2}{pa}=\frac{2 \gamma+n+a-2}{a},
		$$
		then for any (possible infinity) system $\left\{f_{\iota}\right\}_{\iota \in I }$ of orthonormal functions in $L_{k,a}^{2}\left(\mathbb{R}^{n}\right)$ and any coefficients $\left\{n_{\iota}\right\}_{\iota \in I }$ in $\mathbb{C}$, we have
		$$
		\left\|\sum_{\iota\in I} n_{i}\left|e^{i\frac{ t}{a} \|x\|^{2-a}\Delta_k} f_{\iota}\right|^{2} \right\|_{L^q (\mathbb{R}, L_{k,a}^p(\mathbb{R}^n))} \leqslant C\left(\sum_{\iota \in I}\left|n_{\iota }\right|^{\frac{2 p}{p+1}}\right)^{\frac{p+1}{2 p}},
		$$
		where $C>0$ is independent of $\left\{f_{\iota}\right\}_{\iota\in I}$ and $\left\{n_{\iota}\right\}_{\iota \in I}$.
	\end{theorem}

	Apart from the introduction, this paper is organized as follows. 
	\begin{itemize}
		\item In Section \ref{sec2},  we recall some basic definitions and important properties of the Dunkl operator and generalized Laguerre semigroup. 
		\item  In Section \ref{s3}, we explain the relation bewtween the Strichartz estimate for the Schr\"odinger-$\Delta_{k,a}$ equation and the restriction theorem associated with $\Delta_{k,a}$ on certain discrete surface.
		\item  In Section \ref{s4}, we introduce the complex interpolation method and the duality principle in terms of Schatten class and  establish the Schatten boundedness for the Fourier-$\Delta_{k,a}$ extension operator for some $\lambda_0>1$.
		\item  In Section \ref{s5}, we prove   restriction estimates and Strichartz inequalities for systems of orthonormal functions asssociated with the $(k,a)$-generalized Laguerre operator $\Delta_{k,a}$. 
		\item   In Section \ref{s6}, we establish the orthonormal Strichartz inequality for the Dunkl operator $\Delta_k$ by the relation of the kernel between $e^{-it\Delta_{k,a}}$ and $e^{i\frac{a}{t}\|x\|^{2-a}\Delta_k}$.
		\item  In  Section  \ref{s7}, we list some applications to show the results in our paper are more general.
	\end{itemize}

	%
	%
	%
	%
	%
	%
	%
	%
	%
	%
	%

	\section{Preliminary}\label{sec2}
	In this section, we recall some basic definitions and important properties of the Dunkl operator and generalized Laguerre semigroup to make the paper self contained. A complete account of detailed  study  on Dunkl operator  can be found in \cite{ros,ben}. However, we mainly adopt the notation and terminology given in \cite{ben}. The basic ingredient in the theory of Dunkl operators are root systems and finite reflection groups. We start this section by the definition of root system.
	\subsection{Dunkl operator}\label{dunklop} Let $\langle\cdot, \cdot\rangle$  denotes the standard Euclidean scalar product in $\mathbb{R}^{n}$.  For $x \in \mathbb{R}^{n}$, we denote   $\|x\|$ as $\|x\|=\langle x, x\rangle^{1 / 2}$.
	For $\alpha \in \mathbb{R}^{n} \backslash\{0\},$ we denote $r_{\alpha}$ as  the reflection with respect to the hyperplane $\langle \alpha\rangle^{\perp}$ orthogonal to $\alpha$ and is defined by
	$$
	r_{\alpha}(x):=x-2 \frac{\langle \alpha, x\rangle}{\|\alpha\|^{2}} \alpha, \quad x \in \mathbb{R}^{n}.
	$$
	A finite set $\mathcal{R}$ in $\mathbb{R}^{n} \backslash\{0\}$ is  said to be a  root system if the following holds:
	\begin{enumerate}
		\item $ r_{\alpha}(\mathcal{R})=\mathcal{R}$ for all $\alpha \in \mathcal{R}$,
		\item $\mathcal{R} \cap \mathbb{R} a=\{\pm \alpha\}$ for all $\alpha \in \mathcal{R}$.
	\end{enumerate}
	For a given root system $\mathcal{R}$, the subgroup $G \subset O(n, \mathbb{R})$ generated by the reflections $\left\{r_{\alpha} \mid \alpha \in \mathcal{R}\right\}$ is called the finite Coxeter group associated with $\mathcal{R}$.  The dimension of $span \mathcal{R}$ is called the rank of $\mathcal{R}$.   For a   detailed on the theory of finite reflection groups, we refer to  \cite{hum}.  	Let    $\mathcal{R}^+:=\{\alpha\in\mathcal{R}:\langle\alpha,\beta\rangle>0\}$ for some $\beta\in\mathbb{R}^n\backslash\bigcup_{\alpha\in\mathbb{R}}\langle \alpha\rangle^{\perp}$,  be a  fix  positive root system.

	Some  typical examples of such system is the Weyl groups such as the symmetric group $S_{n}$ for the type $A_{n-1}$ root system and the hyperoctahedral group for the type $B_{n}$ root system. In addition, $H_{3}, H_{4}$ (icosahedral groups) and, $I_{2}(n)$ (symmetry group of the regular $n$-gon) are also the Coxeter groups. 

	$A$ multiplicity function for $G$ is a function $k: \mathcal{R} \rightarrow \mathbb{C}$ which is constant on $G$-orbits. Setting $k_{\alpha}:=k(\alpha)$ for $\alpha \in \mathcal{R},$  from the definition of $G$-invariant,  we have $k_{g \alpha}=k_{\alpha}$ for all $g \in G$.  We say $k$ is non-negative if $k_{\alpha} \geq 0$ for all $\alpha \in \mathcal{R}$. Let us denote  $\gamma$ as $\gamma=\gamma(k):=\sum\limits_{\alpha\in\mathcal{R}^+}k(\alpha)$.
	The $\mathbb{C}$-vector space of non-negative multiplicity functions on $\mathcal{R}$ is denoted by $\mathcal{K}^{+}$. For $\xi \in \mathbb{C}^{n}$ and $k \in \mathcal{K}^{+},$ Dunkl  in 1989 introduced a family of first order differential-difference operators $T_{\xi}:= T_{\xi}(k)$  by
	\begin{align}\label{dunkl}
		T_{\xi}(k) f(x):=\partial_{\xi} f(x)+\sum_{\alpha \in \mathcal{R}^{+}} k_{\alpha }\langle \alpha, \xi\rangle \frac{f(x)-f\left(r_{\alpha} x\right)}{\langle \alpha, x\rangle}, \quad f \in C^{1}\left(\mathbb{R}^{n}\right),
	\end{align}
	where $\partial_{\xi}$ denotes the directional derivative corresponding to $\xi$.  The operator $T_\xi$   defined in (\ref{dunkl}) formally known as Dunkl operator and  is one of the most important developments in the theory of special functions associated with root systems \cite{dun}.

	They commute pairwise and are skew-symmetric with respect to the $G$-invariant measure $h_k(x)dx$, where the weight function $h_k(x):=\prod\limits_{\alpha\in \mathcal{R}^{+}}|\langle \alpha, x\rangle|^{2 k_{\alpha}}$ and is homogeneous of degree $2\gamma$. Thanks to the $G$-invariance of the multiplicity function, this definition is independent of the choice of the positive subsystem $\mathcal{R}^{+}$.  In \cite{dun1991},	it is shown    that for any $k\in\mathcal{K}^+$, there is a unique linear isomorphism $V_k$ (Dunkl's intertwining operator) on the space $\mathcal{P}(\mathbb{R}^n)$ of polynomials on $\mathbb{R}^n$ such that
	\begin{enumerate}
		\item $V_k\left(\mathcal{P}_m(\mathbb{R}^n)\right)=\mathcal{P}_m(\mathbb{R}^n)$ for all $m\in\mathbb{N}$,
		\item $V_k|_{\mathcal{P}_0(\mathbb{R}^n)}=id$,
		\item $T_\xi(k)V_k=V_k\partial_\xi$.
	\end{enumerate}
	Here $\mathcal{P}_m(\mathbb{R}^n)$ denotes the space of homogeneous polynomials of degree $m$. 
	
	For any finite reflection group $G$ and   any $k\in\mathcal{K}^+$, R\"{o}sler in \cite{ros} proved that there exists a unique positive Radon probability measure $\rho_x^k$ on $\mathbb{R}^n$ such that
	\begin{equation}\label{239}
		V_kf(x)=\int_{\mathbb{R}^n}f(\xi)d\rho_x^k(\xi).
	\end{equation}
	The measure $\rho_x^k$ depends on $x\in\mathbb{R}^n$ and its support is contained in the ball $B\left(\|x\|\right):=\{\xi\in\mathbb{R}^n: \|\xi\|\leq\|x\|\}$. In view of the Laplace type representation \eqref{239}, Dunkl's intertwining operator $V_k$ can be extended to a larger class of function spaces.
	
	Let $\{\xi_1, \xi_2, \cdots, \xi_n\}$ be an orthonormal basis of $(\mathbb{R}^n , \langle \cdot, \cdot\rangle )$. Then  the Dunkl Laplacian operator $\Delta_k $ is defined as $$\Delta_k=\sum_{j=1}^nT_{\xi_j}(k)^2.$$
	The definition of  $\Delta_k $ is independent of the choice of an orthonormal basis of $\mathbb{R}^n.$ In fact,   one can see that the operator $\Delta_k $ also can be expressed as the following: $$\Delta_kf(x)=\Delta f(x)+\sum_{\alpha \in \mathcal{R}^+}k_\alpha \left\{\frac{2 \langle \nabla f(x),  \alpha \rangle}{\langle \alpha, x\rangle } -\|\alpha\|^2 \frac{f(x)-f(r_\alpha x)}{\langle \alpha, x\rangle^2 } \right\}, \quad f\in C^1(\mathbb{R}^n),$$
	where $\nabla$ and $\Delta$ are      the usual gradiant  and usual Laplacian operator on $\mathbb{R}^n$, respectively.  Note that for $k \equiv 0$, the Dunkl Laplacian operator $\Delta_k$ reduces to the Euclidean Laplacian $\Delta$. 
	\subsection{An orthonormal basis in $L_{k, a}^{2}(\mathbb{R}^n)$}	
	A $k$-harmonic polynomial of degree $m$ is a homogeneous polynomial $p$ on $\mathbb{R}^{n}$ of degree $m$ such that $\Delta_{k} p=0.$ Denote by $\mathcal{H}_{k}^{m}\left(\mathbb{R}^{n}\right)$ the space of $k$-harmonic polynomials of degree $m$. Spherical harmonics (or just $h$-harmonics) of degree $m$ are then defined as the restrictions of $\mathcal{H}_{k}^{m}\left(\mathbb{R}^{n}\right)$ to the unit  sphere $\mathbb{S}^{n-1}$. The spaces $\mathcal{H}_{k}^{m}\left(\mathbb{R}^{n}\right)|_{\mathbb{S}^{n-1}}$ $(m\in\mathbb{N})$ are finite dimensional and orthogonal to each other with respect to the measure $h_k(\omega)d\sigma(\omega)$, where $d \sigma$ be the standard measure on the unit sphere $\mathbb{S}^{n-1}$ and $d_{k}$ the normalizing constant defined by
	\begin{align}\label{110}
		d_{k}:=\left(\int_{S^{n-1}}h_k(\omega)d \sigma(\omega)\right)^{-1}.
	\end{align}
	For $k \equiv 0$, $d_{k}^{-1}$ is the volume of the unit sphere, namely
	$$
	d_{0}=\frac{\Gamma\left(\frac{n}{2}\right)}{2 \pi^{\frac{n}{2}}}.
	$$
	We also have the following   spherical harmonics decomposition 
	\begin{equation}\label{242}
		L^2(\mathbb{S}^{n-1},h_k(\omega)d\sigma(\omega))=\sum_{m\in\mathbb{N}}^\oplus\mathcal{H}_{k}^{m}\left(\mathbb{R}^{n}\right)|_{\mathbb{S}^{n-1}}.
	\end{equation}
	For $a>0$ and $1\leq p<\infty$, let $L_{k, a}^{p}\left(\mathbb{R}^{n}\right)$ be the space of $L^p$-functions on $\mathbb{R}^{n}$ with respect to the weight
	$$
	v_{k, a}(x):=\|x\|^{a-2} h_k(x),
	$$
	which has a degree of homogeneity $a-2+2\gamma$.
	
	Moreover, for the polar coordinates $x=r\omega$ $(r>0,\omega\in\mathbb{S}^{n-1})$, we have
	$$
	v_{k, a}(x)dx=r^{2\gamma+n+a-3}h_k(\omega)drd\sigma(\omega).
	$$
	According to the spherical harmonic decomposition \eqref{242}, there is a unitary isomorphism (see [\citenum{ben}, (3.25)]),
	\begin{equation}
		\sum_{m\in\mathbb{N}}^\oplus\left(\mathcal{H}_{k}^{m}\left(\mathbb{R}^{n}\right)|_{\mathbb{S}^{n-1}}\right)\widehat{\otimes} L^2(\mathbb{R}^+, r^{2\gamma+n+a-3}dr)\overset{\sim}{\rightarrow} L_{k, a}^{2}\left(\mathbb{R}^{n}\right),
	\end{equation}
	where $\widehat{\otimes}$ stands for the Hilbert completion of the tensor product space of two Hilbert spaces.\\
	For any $\mu>-1$, and $\ell,m\in\mathbb{N}$, define the Laguerre polynomial of degree $\mu\in\mathbb{N}$ on $\mathbb{R}^+$ by
	\begin{equation*}
		L^{(\mu)}_\ell(t):=\sum_{j=0}^\ell\frac{\Gamma(\mu+\ell+1)}{(\ell-j)!\Gamma(\mu+j+1)}\frac{(-t)^j}{j!}, \quad t\in \mathbb{R}^+,
	\end{equation*}
	and set $\lambda_{k, a, m}:=\frac{1}{a}(2 m+2 \gamma+n-2)$.
	\begin{proposition}([\citenum{ben}, Proposition 3.15]) For any fixed $m\in\mathbb{N}$, $a>0,$ and a multiplicity funtion $k$ satisfying $\lambda_{k, a, m}>-1$, set
		\begin{equation*}
			\psi_{\ell,m}^{(a)}(r):=\left(\frac{2^{\lambda_{k, a, m}}\Gamma(\ell+1)}{a^{\lambda_{k, a, m}}\Gamma(\lambda_{k, a, m}+\ell+1)}\right)^\frac{1}{2}r^mL^{(\lambda_{k, a, m})}_\ell\left(\frac{2}{a}r^a\right)\exp\left(-\frac{1}{a}r^a\right),\,\forall \ell\in\mathbb{N}.
		\end{equation*}
		Then the set $\left\{\psi_{\ell,m}^{(a)}: \ell\in\mathbb{N}\right\}$ forms an orthonormal basis in $L^2(\mathbb{R}^+, r^{2\gamma+n+a-3}dr)$.
	\end{proposition}
	For every fixed $m\in\mathbb{N}$, we take an orthonormal basis of $\mathcal{H}_{k}^{m}\left(\mathbb{R}^{n}\right)|_{\mathbb{S}^{n-1}}$ as 
	\begin{equation*}
		\left\{Y_j^m: j\in J_m\right\},
	\end{equation*}
	where $J_m=\{1,2,\cdots,\dim\left(\mathcal{H}_{k}^{m}\right)\}$. Then we obtain an orthonormal basis in $L^2_{k,a}(\mathbb{R}^n)$ immeidiately as the following.
	\begin{corollary}([\citenum{ben}, Corollary 3.17]) Suppose that $a>0$ and $k\in\mathcal{K}^+$ satisfies $a+2\gamma+n-2>0$. For each $\ell,m \in\mathbb{N}$ and $j\in J_m$, we set
		\begin{equation*}
			\Phi_{\ell,m,j}^{(a)}(x):=Y_j^m\left(\frac{x}{\|x\|}\right)\psi_{\ell,m}^{(a)}\left(\|x\|\right).
		\end{equation*}
		Then the set $\left\{\Phi_{\ell,m,j}^{(a)}: \ell,m\in\mathbb{N}, j\in J_m\right\}$ forms an orthonormal basis in $L^2_{k,a}(\mathbb{R}^n)$.
	\end{corollary}
	For $\ell, m \in \mathbb{N}$ and $p \in \mathcal{H}_{k}^{m}\left(\mathbb{R}^{n}\right)$, define the function
	$$
	\Phi_{\ell}^{(a)}(p, x):=p(x) L_{\ell}^{\left(\lambda_{k, a, m}\right)}\left(\frac{2}{a}\|x\|^{a}\right) \exp \left(-\frac{1}{a}\|x\|^{a}\right).
	$$
	For $k \in \mathcal{K}^{+}$ and $a>0$ such that $a+2 \gamma+n-2>0,$ the following vector space
	$$
	W_{k, a}\left(\mathbb{R}^{n}\right):=\operatorname{span}\left\{\Phi_{\ell}^{(a)}(p, \cdot) \mid \ell, m \in \mathbb{N}, p \in \mathcal{H}_{k}^{m}\left(\mathbb{R}^{n}\right)\right\}
	$$
	is a dense subspace of $L_{k, a}^{2}\left(\mathbb{R}^{n}\right)$.

	\subsection{$(k,a)$-generalized Laguerre semigroup}
	Let   $a>0$ to be a deformation parameter and  $k$ be a multiplicity function,   then consider  the following differential-difference operator
	$$
	\Delta_{k, a}:=\frac{1}{a}\left(\|x\|^{a}-\|x\|^{2-a} \Delta_{k}\right),
	$$
	where $\|x\|^{a}$ in the right hand side of the formula stands for the multiplication operator by $\|x\|^{a}$. Note that when $a=2$ and $k \equiv 0,$ it reduces to the classical Hermite operator $H:=2\Delta_{0,2}=\|x\|^2-\Delta$ on $L^{2}\left(\mathbb{R}^{n}\right)$, and when $a=1$ and $k \equiv 0,$ it reduces to classical Laguerre operator $L:=\Delta_{0,1}=\|x\|-\|x\|\Delta$ on $L^{2}\left(\mathbb{R}^{n},\|x\|^{-1}dx\right)$. For a general multiplicity function $k$, when $a=2$, it reduces to the Dunkl-Hermite operator $H_k:=2\Delta_{k,2}=\|x\|^2-\Delta_k$ on $L_{k}^{2}\left(\mathbb{R}^{n}\right),$ where $L_{k}^{}\left(\mathbb{R}^{n}\right):=L^{p}\left(\mathbb{R}^{n},h_k(x)dx\right)$, and when $a=1$, it reduces to Dunkl-Laguerre operator $L_k:=\Delta_{k,1}=\|x\|-\|x\|\Delta_k$ on $L^{2}\left(\mathbb{R}^{n},\|x\|^{-1}h_k(x)dx\right)$.
	
	For $a>0$ and $k \in \mathcal{K}^{+}$ such that $a+2 \gamma+n-2>0$, from [\citenum{ben}, Corollary 3.22], $\Delta_{k, a}$ is an essentially self-adjoint operator on $L_{k, a}^{2}\left(\mathbb{R}^{n}\right)$. Moreover, there is no continuous spectrum of $\Delta_{k, a}$ and all the discrete spectra are positive. The orthonormal functions $\left\{\Phi_{\ell,m,j}^{(a)}: \ell,m\in\mathbb{N}, j\in J_m\right\}$ in $L^2_{k,a}(\mathbb{R}^n)$ are eigenfunctions of the operator $\Delta_{k, a}$ corresponding to the eigenvalues $2\ell+\lambda_{k,a,m}+1$ (see  [\citenum{ben}, (3.33 a)]), i.e.,
	$$\Delta_{k, a}\Phi_{\ell,m,j}^{(a)}=(2\ell+\lambda_{k,a,m}+1)\Phi_{\ell,m,j}^{(a)}. $$

	In \cite{ben}, the authors have studied the so-called $(k, a)$-generalized Laguerre semigroup $\mathcal{I}_{k, a}(z)$ with infinitesimal generator $\Delta_{k, a},$ that is
	$$
	I_{k, a}(z):=\exp \left(-z \Delta_{k, a}\right)
	$$
	for $z \in \mathbb{C}$ such that $\operatorname{Re}(z) \geq 0$. Henceforth, we shall denote by $\mathbb{C}^+=\{z\in\mathbb{C}: \operatorname{Re}(z) \geq 0\}$. Suppose a $>0$ and $k \in \mathcal{K}^{+}$ satisfying the condition $a+2 \gamma+n-2>0 .$ Then
	\begin{enumerate}
		\item The map
		$$
		\mathbb{C}^{+} \times L_{k, a}^{2}\left(\mathbb{R}^{n}\right) \longrightarrow L_{k, a}^{2}\left(\mathbb{R}^{n}\right), \quad(z, f) \longmapsto e^{-z \Delta_{k, a}} f
		$$
		is continuous.
		\item For any $p \in \mathcal{H}_{k}^{m}\left(\mathbb{R}^{N}\right)$ and $\ell \in \mathbb{N}, \Phi_{\ell}^{(a)}(p, \cdot)$ is an eigenfunction of the operator $e^{-z \Delta_{k, a}}$, i.e., 
		$$
		e^{-z \Delta_{k, a}} \Phi_{\ell}^{(a)}(p, x)=e^{-z\left(\lambda_{k, a, m}+2 \ell+1\right)} \Phi_{\ell}^{(a)}(p, x).
		$$
		\item   The operator norm    $\left\|e^{-z \Delta_{k, a}}\right\|_{o p}$ is exp $\left(-\frac{1}{a}(2 \gamma+n+a-2) \operatorname{Re}(z)\right)$.
		\item If $\operatorname{Re}(z)>0,$ then $e^{-z \Delta_{k, a}}$ is a Hilbert Schmidt operator.
		\item If $\operatorname{Re}(z)=0,$ then $e^{-z \Delta_{k, a}}$ is an unitary operator.
	\end{enumerate}
	For $\operatorname{Re}(z) \geq 0$,      $\mathcal{I}_{k, a}(z)$ is an integral operator  for all $a>0$ and  has the following expression
	\begin{align}\label{semigroup}
		e^{-z \Delta_{k, a}} f(x)=c_{k, a} \int_{\mathbb{R}^{n}} \Lambda_{k, a}(x, y ; z) f(y) v_{k, a}(y) d y,
	\end{align}
	where  the constant 
	$$
	c_{k, a}=a^{-\frac{2 \gamma+n-2}{a}} \Gamma\left(\frac{2 \gamma+n+a-2}{a}\right)^{-1}d_k,
	$$  and $d_k$ is defined in (\ref{110}). 
	Moreover, a series expansion for the kernel $\Lambda_{k, a} $ can be found in [\citenum{ben}, Theorem 4.20]. Since the series is expressed compactly for $a = 1, 2$, throughout the paper  we will also assume that  $a = 1, 2$.

	Recall the expression of the kernel $\Lambda_{k, a}(x, y ; z)$ in the expression (\ref{semigroup}) given by $\Lambda_{k, a}(x, y ; z):=V_{k}^{\eta} h_{a}(r, s ; z ;\langle\omega, \cdot\rangle)(\eta)$ (using the polar coordinate $x=r \omega, y=s \eta$ ), where  for all $\zeta \in[-1,1]$ with parameters $r, s>0$ and $z \in \mathbb{C}^{+} \backslash  i\pi \mathbb{Z} :$
	$$
	\begin{aligned}
		h_{a}(r, s ; z ; \zeta):=& \frac{\exp \left(-\frac{1}{a}\left(r^{a}+s^{a}\right) \operatorname{coth} z\right)}{(\sinh z)^{\frac{2 \gamma+n-1}{a}}} \\
		& \times\left\{\begin{array}{ll}
			\Gamma\left(\gamma+\frac{n-1}{2}\right) \widetilde{I}_{\gamma+\frac{n-3}{2}}\left(\frac{\sqrt{2}(r s)^{1 / 2}}{\sinh z}(1+\zeta)^{1 / 2}\right), & a=1, \\
			\exp \left(\frac{r s \zeta}{\sinh z}\right), & a=2.
		\end{array}\right.
	\end{aligned}
	$$
	Here $\widetilde{I}_{\lambda}(w)=\left(\frac{w}{2}\right)^{-\lambda} I_{\lambda}(w)$ is the (normalized) modified Bessel function of the first kind, $V_{k} $ is the Dunkl  intertwining operator, and   the superscript in $V_{k}^{\eta}$ denotes the corresponding  variable. We point out that, when  $k \equiv 0 , $   $\Lambda_{k, a}(x, y ; z)=h_{a}(r, s ; z ;\langle\omega, \eta\rangle)$.

	For $a=1,2$, it follows that the kernel $\Lambda_{k, a}(x, y ; i \mu)$ of $e^{-i \mu \Delta_{k, a}}$ satisfies   
	\begin{enumerate}
		\item $
		\Lambda_{k, a}(x, y ;-i \mu)=\overline{\Lambda_{k, a}(x, y ; i \mu)},
		$
		\item $
		\Lambda_{k, a}(x, y ; i(\mu+\pi))=e^{-i \pi\left(\frac{2 \gamma+n+a-2}{a}\right)} \Lambda_{k, a}\left((-1)^{\frac{2}{a}} x, y ; i\mu\right),
		$
	\end{enumerate}
	for all $\mu \in \mathbb{R} \backslash \pi \mathbb{Z}$, which follows that the $L_{k,a}^p(\mathbb{R}^n)$ norm of $e^{-i \mu \Delta_{k, a}}f$ is $\pi$-periodic in $t$ and thus determined by its values for $t\in (-\frac{\pi}{2},\frac{\pi}{2})$. 
	
	
	Moreover, the kernel $\Lambda_{k, a}(x, y ; z)$ of $e^{-z \Delta_{k, a}}$ satisfies the following upper bound estimates (see  [\citenum{ben},Proposition 4.26]).
	\begin{proposition} \label{eq200} 
		For $a=1,2$, the function $\Lambda_{k, a}(x, y ; z)$ satisfies the following inequalities:
		\begin{enumerate}
			\item For $\operatorname{Re}(z)>0$, there exists a constant $C>0$ depending on $z$ such that 
			$$\left|\Lambda_{k, a}(x, y ; z)\right|\leq \frac{1}{|\sinh{z}|^\frac{2\gamma+n+a-2}{a}}\exp{\left(-C(\|x\|^a+\|y\|^a)\right)}.$$
			\item For $z=\epsilon+i\mu$ such that $\epsilon\geq0$ and $\mu \in \mathbb{R} \backslash \pi \mathbb{Z}$, we have
			$$\left|\Lambda_{k, a}(x, y ; z)\right|\leq \frac{1}{|\sin{\mu}|^\frac{2\gamma+n+a-2}{a}}.$$
		\end{enumerate}
	\end{proposition}
	Now, we consider the Cauchy problem for the free Schr\"odinger equation associted with the $(k,a)$-generalized Laguerre operator $\Delta_{k,a}$, namely
	\begin{align}\label{100}
		\left\{  \begin{array}{ll}i \partial_{t} u(t, x)- \Delta_{k, a} u(t, x)=0, & (t, x) \in \mathbb{R} \times  \mathbb{R}^n, \\ u(0,x)=f(x)\in L_{k, a}^{2}\left(\mathbb{R}^{n}\right).\end{array}\right.
	\end{align}
	Then $u(t, x)=e^{-i t \Delta_{k, a}} f(x)$ is the solution of the above system. As the solution of the problem \eqref{100} is $\pi$-periodic in $t$, we introduce the mixed normed space $L^{q}\left( ( -\frac{\pi }{2}, \frac{\pi }{2}), L_{k, a}^{p}(\mathbb{R}^n)\right)$ which is the set of measure functions $h$ on $(-\frac{\pi }{2}, \frac{\pi }{2})\times\mathbb{R}^n$ such that
	\begin{equation*}
		\|h\|_{L^{q}\left( ( -\frac{\pi }{2}, \frac{\pi }{2}), L_{k, a}^{p}(\mathbb{R}^n)\right)}  :=\left\|\|h(t,\cdot)\|_{L_{k, a}^{p}(\mathbb{R}^n)}\right\|_{L^{q}(-\frac{\pi }{2}, \frac{\pi }{2})}.  
	\end{equation*}
	A pair $(p, q)$ is called admissible if $\left(\frac{1}{p}, \frac{1}{q}\right)$ belongs to the trapezoid
	$$
	\frac{1}{2}\left(\frac{2 \gamma+n-2}{2 \gamma+n+a-2}\right)<\frac{1}{p} \leq \frac{1}{2} \text { and } \frac{1}{2} \leq \frac{1}{q} \leq 1
	$$
	or
	$$
	0 \leq \frac{1}{q}<\frac{1}{2} \text { and } \frac{1}{q} \geq\left(\frac{2 \gamma+n+a-2}{a}\right)\left(\frac{1}{2}-\frac{1}{p}\right).
	$$
	The Strichartz estimates for the solution of \eqref{100} is proved by \cite{Ratna3} and may be stated as:
	\begin{theorem}[Strichartz inequality for a single function]\label{single}
		Suppose $a=1,2$ and $k$ is a non-negative multiplicity function such that
		$$
		a+2 \gamma+n-2>0.
		$$
		Let $(p, q)$ be an admissible pair  and $u=e^{-i t \Delta_{k, a}} f$ be the solution to the homogeneous problem (\ref{100}) with $f \in L_{k, a}^{2}\left(\mathbb{R}^{n}\right) .$ Then we have the estimate
		$$
		\left\|e^{-i t \Delta_{k, a}} f\right\|_{L^{q}\left( ( -\frac{\pi }{2}, \frac{\pi }{2}), L_{k, a}^{p}(\mathbb{R}^n)\right)} \leq C\|f\|_{L_{k, a}^{2}\left(\mathbb{R}^{n}\right)}.
		$$
	\end{theorem}
	Next we turn our attention to the free Schr\"odinger equation with respect to the differential-difference part of $\Delta_{k,a}$
	$$
	\left\{\begin{array}{ll}
		i \partial_{t} w(t, x)+\frac{1}{a}\|x\|^{2-a} \Delta_{k} w(t, x)=0, & (t, x) \in \mathbb{R} \times\mathbb{R}^n, \\
		w(0,x)=f(x),
	\end{array}\right.
	$$
	Then  $e^{i \frac{t}{a}\|x\|^{2-a} \Delta_{k}} f(x)$ is the solution of the above Schr\"odinger equation. 
	Let $\Gamma_{k, a}(x, y ; i t)$  be the kernel of $e^{i \frac{t}{a}\|x\|^{2-a} \Delta_{k}}$. Using the change of variable $s=\tan (t)$ with $t \in(-\pi / 2, \pi / 2),$ we get
	\begin{align*}
		\Lambda_{k, a}(x, y ; i \tan^{-1} s)&=c_{k, a}^{-1}\left(1+s^{2}\right)^{\frac{2 y+n+a-2}{2 a}} \exp \left(-i s \frac{\|x\|^{a}}{a}\right) \Gamma_{k, a}\left(\left(1+s^{2}\right)^{\frac{1}{a}} x, y ; i s\right),\\
		e^{-i \tan^{-1}(s) \Delta_{k, a}} f(x)&=\left(1+s^{2}\right)^{\frac{2 \gamma+n+a-2}{2 a}} e^{-i s \frac{\|x\|^{a}}{a}} e^{i \frac{s}{a}\|x\|^{2-a} \Delta_{k}} f\left(\left(1+s^{2}\right)^{\frac{1}{a}} x\right).
	\end{align*}
	for any $f \in L_{k, a}^{2}\left(\mathbb{R}^{n}\right)$.	
	
	\begin{theorem}[\cite{Ratna3}]\label{111}
		Suppose $a=1,2$ and $k$ is a non-negative multiplicity function such that
		$$
		a+2 \gamma+n-2>0.
		$$
		If $1\leq p, q\leq\infty$ satisfy  $$\left(\frac{1}{2}-\frac{1}{p}\right)\frac{2 \gamma+n+a-2}{a}-\frac{1}{q}=0,$$
		then for all $f \in L_{k, a}^{2}\left(\mathbb{R}^{n}\right)$ we have
		\begin{align*}
			\|e^{i \frac{t}{a}\|x\|^{2-a} \Delta_{k}} f\|_{L^{q}\left( ( 0, +\infty),L_{k,a}^p(\mathbb{R}^n)\right)}&=\|e^{-i t \Delta_{k, a}} f\|_{L^{q}\left( ( 0, \frac{\pi }{2}),L_{k,a}^p(\mathbb{R}^n)\right)},\\
			\|e^{i \frac{t}{a}\|x\|^{2-a} \Delta_{k}} f\|_{L^{q}\left( (-\infty,0),L_{k,a}^p(\mathbb{R}^n)\right)}&=\|e^{-i t \Delta_{k, a}} f\|_{L^{q}\left( (-\frac{\pi }{2},0),L_{k,a}^p(\mathbb{R}^n)\right)}.
		\end{align*}
		Moreover, we have
		\begin{equation*}
			\|e^{i \frac{t}{a}\|x\|^{2-a} \Delta_{k}} f\|_{L^{q}\left( \mathbb{R},L_{k,a}^p(\mathbb{R}^n)\right)}\leq C\|f\|_{L_{k, a}^{2}\left(\mathbb{R}^{n}\right)}.
		\end{equation*}
	\end{theorem}

	\subsection{Schatten class} 		If \(\mathcal{H}\) is a complex,   separable Hilbert space, a linear compact operator \(A : \mathcal{H} \rightarrow \mathcal{H}\) belongs to the $r$-Schatten-von Neumann class \(\mathcal{G}^{r}(\mathcal{H})\) if
	$$
	\sum_{n=1}^{\infty}\left(s_{n}(A)\right)^{r}<\infty,
	$$where \(s_{n}(A)\) denote the singular values of \(A,\) i.e. the eigenvalues of \(|A|=\sqrt{A^{*} A}\)
	with multiplicities counted. For $1 \leq r<\infty,$ the Schatten space $\mathcal{G}^{r}(\mathcal{H})$ is defined as the space of all compact operators $A$ on $\mathcal{H}$ such that $\displaystyle \sum_{n =1}^\infty  \left(s_{n}(A)\right)^{r}<\infty$.
	
	For \(1 \leq r<\infty\),  the class \(\mathcal{G}^{r}(\mathcal{H})\) is  a Banach space
	endowed with the norm
	$$
	\|A\|_{\mathcal{G}^{r}}=\left(\sum_{n=1}^{\infty}\left(s_{n}(A)\right)^{r}\right)^{\frac{1}{r}}.
	$$
	For  \(0<r<1\), the $\|\cdot\|_{\mathcal{G}^{r}}$ as above only defines a quasi-norm with respect to which
	\(\mathcal{G}^{r}(\mathcal{H})\) is complete. An operator belongs to the class \(\mathcal{G}^{1}(\mathcal{H})\) is known as {\it Trace class} operator. Also, an operator belongs to   \(\mathcal{G}^{2}(\mathcal{H})\) is known as  {\it Hilbert-Schmidt} operator.
	
	\subsection{An analytic family of operators}
	Let us first recall that a family of operators $\{T_z\}$ on $\mathbb{T}\times \mathbb{R}^n$ defined in a strip $\alpha\leq \operatorname{Re}(z)\leq \beta$ in the complex plane is analytic in the sense of Stein if it has the following properties:
	
	\begin{enumerate}
		\item For each $z: \alpha\leq \operatorname{Re}(z)\leq \beta$, $T_z$ is a linear transformation of simple functions on $\mathbb{T}\times \mathbb{R}^n$ (i.e., functions that take on only a finite number of nonzero values on sets of finite measure on $\mathbb{T}\times \mathbb{R}^n$ to measurable functions on $\mathbb{T}\times \mathbb{R}^n$).
		
		\item For all simple functions $F, G$ on $\mathbb{T}\times \mathbb{R}^n$, the map $z\rightarrow \langle G, T_zF\rangle$ is analytic in $a<\operatorname{Re}(z)<b$ and continuous in $a\leq \operatorname{Re}(z)\leq b$.
		
		\item  Moreover, $\sup_{\alpha\leq\lambda\leq \beta}\left|\langle G, T_{\lambda+is}F\rangle\right|\leq C(s)$ for some $C(s)$ with at most a (double) exponential growth in $s$.
	\end{enumerate}
	\section{Restriction theorem for the $(k, a)$-generalized Laguerre operator}\label{s3}

	For $a=1,2$, let $S$ be the discrete surface $S=\{(\nu,\ell, m , j)\in \mathbb{Z}\times \mathcal{A}: \nu=2\ell+\frac{2m}{a}\}$ with respect to the counting measure.  Choose 
	$$
	\hat{F}(\nu,\ell, m , j)=\left\{\begin{array}{ll}       {\hat{f}(\ell, m , j )}, & {\text { if } \nu=2\ell+\frac{2m}{a},}\\{0}, &~ {\text {otherwise,} }\end{array} \right.
	$$ for some  measurable function $f$ on  $\mathbb{R}^n.$
	Then   for any $f\in L^2_{k,a}(\mathbb{R}^n)$, by the Plancherel formula, we have
	\begin{align}\label{eq10}\nonumber
		\|F\|_{L^{2}\left(\mathbb{T}, L^2_{k,a}(\mathbb{R}^n)\right)}&=\sqrt{2 \pi}\|\{\hat{F}(\nu,\ell, m , j)\}\|_{\ell^{2}\left(S\right)} \\
		&=\sqrt{2 \pi}\|\{\hat{f}(\ell, m , j )\}\|_{\ell^{2}\left(\mathcal{A}\right)} =\sqrt{2 \pi}\|f\|_{ L^2_{k,a}(\mathbb{R}^n)}.
	\end{align}
	Thus for $F \in L^{2}\left(\mathbb{T}, L^2_{k,a}(\mathbb{R}^n)\right)$, we get
	\begin{align}\label{eq11}\nonumber
		\mathcal{E}_S(\{\hat{F}(\nu,\ell, m , j)\})(t,x)&= \sum_{(\nu,\ell, m , j) \in S} \hat{F}(\nu,\ell, m , j)\Phi_{\ell,m,j}^{(a)}(x)e^{-i t \nu}\\\nonumber
		&= \sum_{(\ell, m , j ) \in \mathcal{A}} \hat{f}(\ell, m , j )\Phi_{\ell,m,j}^{(a)}(x)e^{-i t\left(2\ell+\frac{2m}{a}\right) } \\\nonumber
		&= \sum_{(\ell, m , j ) \in \mathcal{A}} 
		\langle f, \Phi_{\ell,m,j}^{(a)}\rangle  \Phi_{\ell,m,j}^{(a)}(x)e^{-i t\left(2\ell+\frac{2m}{a}\right) }\\\nonumber
		&= e^{i t\left(1+\frac{1}{a}(2\gamma+n-2)\right)} \ \sum_{(\ell, m , j ) \in \mathcal{A}}   e^{-i t\left(2\ell+1+\frac{1}{a}(2m+2\gamma+n-2)\right) } \langle f, \Phi_{\ell,m,j}^{(a)}\rangle  \Phi_{\ell,m,j}^{(a)}(x)\\
		&= e^{i t\left(1+\frac{1}{a}(2\gamma+n-2)\right)}e^{-i t\Delta_{k, a}} f(x).
	\end{align}
	Now once we assume that  Problem 2 holds, i.e.,  $\mathcal{E}_{S}$ is bounded from $\ell^{2}(S)$ to $L^{q'}\left(\mathbb{T}, L^{p'}_{k,a}(\mathbb{R}^n)\right)$  for some $1 \leq p, q \leq 2$, then from from (\ref{eq10}) and (\ref{eq11}),  the Strichartz inequality follows   as
	$$
	\begin{aligned}
		\left\|e^{-i t\Delta_{k, a}} f\right\|_{L^{q'}\left(\mathbb{T}, L^{p'}_{k,a}(\mathbb{R}^n)\right)} &=\left\|\ 	\mathcal{E}_S(\{\hat{F}(\nu,\ell, m , j)\})\right\|_{L^{q'}\left(\mathbb{T}, L^{p'}_{k,a}(\mathbb{R}^n)\right)} \\
		& \leq C\|\{\hat{F}(\nu,\ell, m , j)\}\|_{\ell^{2}(S)} \\
		&=C\sqrt{2 \pi}\|f\|_{ L^2_{k,a}(\mathbb{R}^n)}.
	\end{aligned}
	$$
	Thus from the above, we can conclude that the Strichartz inequality for the solution of   (\ref{100}) associated with the $(k, a)$-generalized Laguerre  operator holds if only if the restriction theorem holds on the particular  surface $S=\{(\nu,\ell, m , j)\in \mathbb{Z}\times\mathcal{A} : \nu=2\ell+\frac{2m}{a}\}$.

	By Theorem \ref{single}, we have the following restriction estimates.
	\begin{theorem}[Restriction theorem for a single function]\label{Restriction-single}
		Let $S=\{(\nu,\ell, m , j)\in \mathbb{Z}\times \mathcal{A}: \nu=2\ell+\frac{2m}{a}\}$ be the discrete surface.  Then  under the same hypotheses as in Theorem \ref{single}, we have
		$$ \|	\mathcal{E}_S(\{\hat{F}(\nu,\ell, m , j)\})\|_{L^{q}\left(( -\frac{\pi }{2}, \frac{\pi }{2}), L^{p}_{k,a}(\mathbb{R}^n)\right) }\leq C \|	 \{\hat{F}(\nu,\ell, m , j)\})\|_{\ell^2(S)},$$
		with $C>0$ is a constant independent of $\{\hat{F}(\nu,\ell, m , j)\}_{(\nu,\ell, m , j)\in S}$.
	\end{theorem}  
	
	\section{The Schatten boundedness of $\mathcal{T}_S$}\label{s4}
	\subsection{The complex interpolation method and the duality principle} In order to generalize the restriction and Strichartz estimates involving a single function to the system of orthonormal functions, we need to introduce the complex interpolation method and the duality principle lemma in our context. We refer to Proposition 1 and Lemma 3 of \cite{FS} with appropriate modifications to obtain the following two results:
	\begin{proposition}\label{CH2prop1}
		Let $\{T_{z}\}$ be an analytic family of operators on $\mathbb{T}\times \mathbb{R}^n $ in the sense of Stein defined on the strip $-\lambda_0\leq  \operatorname{Re}(z)\leq 0$
		for some $\lambda_0 > 1$. Assume that we have the following bounds
		\begin{align}\label{CH2expo} 
			\begin{cases}
				\left\|T_{i s}\right\|_{L^{2}\left(\mathbb{T}, L^2_{k,a}(\mathbb{R}^n)\right) \rightarrow  L^{2}\left(\mathbb{T}, L^2_{k,a}(\mathbb{R}^n)\right) } \leq M_{0} e^{a|s|},\\\left\|T_{-\lambda_{0}+i s}\right\|_{L^{1}\left(\mathbb{T}, L^1_{k,a}(\mathbb{R}^n)\right) \rightarrow  L^{\infty}\left(\mathbb{T}, L^\infty_{k,a}(\mathbb{R}^n)\right) } \leq M_{1} e^{b|s|},
			\end{cases}	
		\end{align}
		for all $ s \in \mathbb{R}$, for some $a, b ,M_{0}, M_{1} \geq 0 $. Then, for all  $W_{1}, W_{2} \in L^{2 \lambda_{0}}\left(\mathbb{T}\times \mathbb{R}^n\right)$,
		the operator $W_{1} T_{-1} W_{2}$ belongs to ${\mathcal{G}^{2 \lambda_{0}}\left(L^{2}\left(\mathbb{T}, L^2_{k,a}(\mathbb{R}^n)\right)  \right)}$ and we have the estimate
		\begin{align}\label{CH2510}\nonumber
			&	\left\|W_{1} T_{-1} W_{2}\right\|_{\mathcal{G}^{2 \lambda_{0}}\left(L^{2}\left(\mathbb{T}, L^2_{k,a}(\mathbb{R}^n)\right)  \right)}\\& \leq M_{0}^{1-\frac{1}{\lambda_{0}}} M_{1}^{\frac{1}{\lambda_{0}}}\left\|W_{1}\right\| _{L^{2\lambda_0}\left(\mathbb{T}, L^{2\lambda_0}_{k,a}(\mathbb{R}^n)\right)  } \left\|W_{2 }\right\| _{L^{2\lambda_0}\left(\mathbb{T}, L^{2\lambda_0}_{k,a}(\mathbb{R}^n)\right)  }.
		\end{align}
		
	\end{proposition}

	\begin{lemma}[Duality principle]\label{duality-principle} Let $\lambda\geq1$ and $S\subset\mathbb{Z}\times\mathcal{A}$ be a discrete surface. Assume that $A$ is a bounded linear operator from $\ell^2(S)$ to $L^{q^\prime}(\mathbb{T}, L_{k,a}^{p^\prime}(\mathbb{R}^n))$ for some $p,q\geq 1$. Then the following statements are equivalent.
		
		$(1)$ There is a constant $C>0$ such that for all $W\in L^\frac{2q}{2-q}(\mathbb{T}, L_{k,a}^\frac{2p}{2-p}(\mathbb{R}^n))$
		\begin{equation}\label{Schatten}
			\big\|WAA^*\overline{W}\big\|_{\mathcal{G}^\lambda(L^2(\mathbb{T}, L_{k,a}^2(\mathbb{R}^n)))}\leq C\big\|W\big\|_{L^\frac{2q}{2-q}(\mathbb{T}, L_{k,a}^\frac{2p}{2-p}(\mathbb{R}^n))}^2,
		\end{equation}
		where the function $W$ is interpreted as an operator which acts by multiplication.
		
		$(2)$ There is a constant $C'>0$ such that for any orthonormal system $\left\{f_\iota\right\}_{\iota\in{I}}$ in $L^2_{k,a}(\mathbb{R}^n)$ and any sequence $\{n_\iota\}_{\iota\in{I}}$ in $\mathbb{C}$
		\begin{equation*}
			\bigg\|\sum_{\iota\in I}n_\iota|Af_\iota|^2\bigg\|_{L^{\frac{q^\prime}{2}}(\mathbb{T}, L_{k,a}^{\frac{p^\prime}{2}}(\mathbb{R}^n))}\leq C'\bigg(\sum_{\iota\in I}|n_\iota|^{\lambda^\prime}\bigg)^{\frac{1}{\lambda^\prime}}.
		\end{equation*}		
	\end{lemma}
	\begin{remark}
		Proposition \ref{CH2prop1} and Lemma  \ref{duality-principle} is also valid in the domain $(-\frac{\pi}{2},\frac{\pi}{2})\times\mathbb{R}^n$.
	\end{remark}

	\subsection{On general discrete surface}In this subsection we establish the Schatten boundedness of $\mathcal{T}_S$ on general discrete surface. Let $S=\{ (\nu,\ell, m , j)\in \mathbb{Z}\times \mathcal{A}: R(\ell, m , j , \nu)=0\}$ be a discrete surface,  where $R(\nu,\ell, m , j)$ is a polynomial of degree one,  with respect to the counting measure.
	
	For some $\lambda_0>1$ and $-\lambda_0\leq  \operatorname{Re}(z)\leq 0$, consider the analytic family of generalized functions
	\begin{align}\label{CH2gf}
		G_z(\nu,\ell, m , j)=\psi(z)R(\nu,\ell, m , j)_+^z,
	\end{align}
	where $\psi(z)$ is an appropriate analytic function with a simple zero at $z=-1$ with  exponential growth at infinity when $\operatorname{Re}(z)=0 $ and 
	\begin{align}\label{plus}
		R(\nu,\ell, m , j)_+^z=\begin{cases}
			R(\nu,\ell, m , j)^z&\text{ for }  R(\nu,\ell, m , j)>0,\\0&\text{ for }R(\nu,\ell, m , j)\leq0.
	\end{cases} 	\end{align}
	Restricting the Schwartz class function $\Phi$ on $\mathbb{Z}\times \mathcal{A},$ we have $$	\left\langle G_{z}, \Phi\right\rangle:=\psi(z) \sum_{r\in \mathbb{Z}}r_+^z \sum_{\{(\nu,\ell, m , j):R(\nu,\ell, m , j)=r \}}\Phi(\nu,\ell, m , j) ,$$ where $r_+^z$ is defined as in (\ref{plus}). Using one dimensional analysis of $r_+^z$ (see \cite{sh}), we have
	\begin{align*}
		\displaystyle\lim_{z\to -1}\left\langle G_{z}, \Phi\right\rangle
		=\sum_{(\nu,\ell, m , j)\in S}\Phi(\nu,\ell, m , j), \end{align*}
	and this   ensures that $G_{-1}\equiv \delta_S.$
	
	For $-\lambda_0\leq  {\rm Re}(z)\leq 0$, define the analytic family of operators $T_z$ acting on Schwartz class functions on $\mathbb{T}\times \mathbb{R}^n$ by
	$$T_z F(t,x) =  \sum_{(\nu,\ell, m , j)\in \mathbb{Z}\times \mathcal{A}} \hat{F}(\nu,\ell, m , j) G_z(\nu,\ell, m , j) \Phi_{\ell,m,j}^{(a)}(x)e^{-i t \nu}.$$ 
	Then $\{T_z\}$  is an analytic in the sense of Stein defined on the strip $-\lambda_0\leq   {\rm Re}(z)\leq 0$ for some $\lambda_0>1.$ Moreover, we have the identity $\mathcal{T}_S=T_{-1}$ and
	\begin{align}
		&\quad T_z F( t,x) \notag\\&=  \sum_{(\nu,\ell, m , j)\in \mathbb{Z}\times \mathcal{A}} \hat{F}(\nu,\ell, m , j) G_z(\nu,\ell, m , j) \Phi_{\ell,m,j}^{(a)}(x)e^{-i t \nu}\notag\\
		&= \frac{1}{2 \pi}\sum_{(\nu,\ell, m , j)\in \mathbb{Z}\times \mathcal{A}} \left[  \int_{\mathbb{R}^{n}} \int_{\mathbb{T}} F(\tau,y) \Phi_{\ell,m,j}^{(a)}(y) e^{i \tau \nu}\, d t\tau \,v_{k, a}(y) dy\right]G_z(\nu,\ell, m , j) \Phi_{\ell,m,j}^{(a)}(x)e^{-i t \nu}\notag\\
		&= \frac{1}{2 \pi}   \int_{\mathbb{R}^{n}} \int_{\mathbb{T}}  \left[  \sum_{(\nu,\ell, m , j)\in \mathbb{Z}\times \mathcal{A}}  \Phi_{\ell,m,j}^{(a)}(y) \Phi_{\ell,m,j}^{(a)}(x)   G_z(\nu,\ell, m , j)  e^{-i  \nu(t-\tau)}\right]  F(\tau,y) \, d\tau \,v_{k, a}(y) dy \notag\\
		&= \int_{\mathbb{R}^n }  (K_z(\cdot,x, y)*F(\cdot,y))(t)   \,v_{k, a}(y) dy,\label{CH222-22}
	\end{align}
	where $$ K_z(t, x, y) = \frac{1}{2 \pi}     \sum_{(\nu,\ell, m , j)\in \mathbb{Z}\times \mathcal{A}}  \Phi_{\ell,m,j}^{(a)}(y) \Phi_{\ell,m,j}^{(a)}(x)   G_z(\nu,\ell, m , j)  e^{-i  \nu t}.$$  When $ {\rm Re}(z)=0$, we have
	\begin{align}\label{CH2twoo}
		\|T_{is} \|_{L^{2}\left(\mathbb{T}, L^2_{k,a}(\mathbb{R}^n)\right) \rightarrow  L^{2}\left(\mathbb{T}, L^2_{k,a}(\mathbb{R}^n)\right) } =\left\|G_{i s}\right\|_{\ell^{\infty}\left(\mathbb{Z}\times \mathcal{A} \right)}\leq \left|  \psi(is)\right|.	\end{align}
	Again an application of H\"older and Young inequalities in \eqref{CH222-22} gives \begin{eqnarray}\label{CH21inf}
		\left\|T_{z}\right\|_{L^{1}\left(\mathbb{T}, L^1_{k,a}(\mathbb{R}^n)\right) \rightarrow  L^{\infty}\left(\mathbb{T}, L^\infty_{k,a}(\mathbb{R}^n)\right) } \leq  \sup_{x,y\in \mathbb{R}^n, t\in \mathbb{T}}|K_z(t, x, y)|  .
	\end{eqnarray} 
	Using Proposition \ref{CH2prop1}, we   obtain the  following Schatten boundedness of the form (\ref{Schatten}).
	
	\begin{lemma}\label{CH2diagggg}  
		Let $n \geq 1$ and $S \subset \mathbb{Z}\times \mathcal{A}$ be a discrete  surface. Suppose that  for each $x,y\in \mathbb{R}$ and $ t\in \mathbb{T}$,    $|K_z(t,x,y)|$ is   uniformly bounded by a constant  $C(s) $ with   at most exponential growth in $s$  at  when  $z=-\lambda_{0}+is$ for some $\lambda_0>1$. Then  for all $W_1, W_2 \in L^{2\lambda_0}\left(\mathbb{T}, L^{2\lambda_0}_{k,a}(\mathbb{R}^n)\right) ,$ the operator $W_{1} \mathcal T_{S} W_{2}=W_{1}   T_{-1} W_{2}$
		belongs to $\mathcal{G}^{2 \lambda_{0}} (L^{2} (\mathbb{T}, L^2_{k,a}(\mathbb{R}^n) )    )$ and we have the estimate $$	\left\|W_{1} \mathcal T_{S} W_{2}\right\|_{\mathcal{G}^{2 \lambda_{0}}\left(L^{2}\left(\mathbb{T}, L^2_{k,a}(\mathbb{R}^n)\right)  \right)} \leq  C \left\|W_{1}\right\| _{L^{2\lambda_0}\left(\mathbb{T}, L^{2\lambda_0}_{k,a}(\mathbb{R}^n)\right)  } \left\|W_{2 }\right\| _{L^{2\lambda_0}\left(\mathbb{T}, L^{2\lambda_0}_{k,a}(\mathbb{R}^n)\right)  }.$$
	\end{lemma}
	\begin{remark}
		Lemma \ref{CH2diagggg} also holds true in the domain $(-\frac{\pi}{2},\frac{\pi}{2})\times\mathbb{R}^n$.
	\end{remark}
	
	\subsection{On the particular surface $S=\{(\nu, \ell, m , j)\in \mathbb{Z}\times \mathcal{A}: \nu=2\ell+\frac{2m}{a}\}$ }
	In this subsection we   work on this particular   discrete surface $S=\{(\nu, \ell, m , j)\in \mathbb{Z}\times \mathcal{A}: \nu=2\ell+\frac{2m}{a}\}$  with respect to the counting measure.   
	
	Setting $\psi(z)=\frac{1}{\Gamma(z+1)}$ and $ 	R(\nu, \ell, m , j)=\nu-2\ell-\frac{2m}{a}$ in (\ref{CH2gf}),  then we have
	\begin{align*}
		G_z(\nu, \ell, m , j)=\frac{1}{\Gamma(z+1)}\left(\nu-(2\ell+\frac{2m}{a})\right)_+^z
	\end{align*} and
	\begin{align*}		 
		\lim_{z\to -1}\left\langle G_{z}, \Phi\right\rangle&=\lim_{z\to -1}\frac{1}{\Gamma(z+1)}\sum_{r\in \mathbb{Z}}r_+^z \sum_{\{(\nu, \ell, m , j):\nu-(2\ell+\frac{2m}{a})=r\}}\Phi(\nu, \ell, m , j)\\
		&=\sum_{(\nu, \ell, m , j)\in S}\Phi(\nu, \ell, m , j).
	\end{align*}
	Thus $G_{-1}=\delta_S$ and
	\begin{align}\label{CH22-2}T_z F( t,x) 
		&= \int_{\mathbb{R}^n }  (K_z(\cdot, x, y)*F(\cdot, y))(t)   \,v_{k, a}(y) dy,
	\end{align}
	where \begin{align}\label{CH2kernl}\nonumber
		K_z(t, x, y)&
		=\frac{1}{2 \pi}     \sum_{(\nu, \ell, m , j)\in \mathbb{Z}\times \mathcal{A}}  \Phi_{\ell,m,j}^{(a)}(y) \Phi_{\ell,m,j}^{(a)}(x)   G_z(\nu, \ell, m , j) e^{-i  \nu t}\\\nonumber
		&=\frac{1}{2\pi \Gamma(z+1)}  \sum_{(\ell, m , j , \nu)\in \mathbb{Z}\times \mathcal{A}}  \Phi_{\ell,m,j}^{(a)}(y) \Phi_{\ell,m,j}^{(a)}(x)   \left(\nu-(2\ell+\frac{2m}{a})\right)_+^z e^{-i  \nu t}\\\nonumber
		&=\frac{e^{i t\left(1+\frac{1}{a}(2\gamma+n-2)\right)} }{2\pi \Gamma(z+1)}  \ \sum_{(\ell, m , j ) \in \mathcal{A}}   e^{-i t\left(2\ell+1+\frac{1}{a}(2m+2\gamma+n-2)\right) } \Phi_{\ell,m,j}^{(a)}(y) \Phi_{\ell,m,j}^{(a)}(x)\sum_{r=0 }^\infty r_+^z e^{-i rt }\\
		&=\frac{e^{i t\left(1+\frac{1}{a}(2\gamma+n-2)\right)} }{2\pi \Gamma(z+1)}  \Lambda_{k, a}(x, y ; t)  \sum_{r=0 }^\infty r_+^z e^{-i rt }, 
	\end{align}
	where the last equality comes from the spectral decomposition of $e^{-it\Delta_{k, a}}$. In order to obtain a uniformly estimate of the kernel $	K_z $, we need to estimate $ \displaystyle\sum_{r=0 }^\infty r_+^z e^{-i rt }$. 
	\begin{lemma}[\cite{shyam}]\label{CH2515}
		Let $-\lambda_0\leq \operatorname{Re}(z)\leq 0$ for some $\lambda_0>1$. Then the series
		$ \sum_{r=0 }^\infty r_+^z e^{-i rt }$ is the Fourier series of an integrable function on $[-\pi, \pi]$ which is of class $C^\infty $ on $[-\pi, \pi]\setminus  \{0\}.$ Near  the origin this function has the same singularity as the function whose values are $ \Gamma(z+1)(it)^{-z-1},$ i.e.,
		\begin{align}\label{CH2eq4} \sum_{r=0 }^\infty r_+^z e^{-i rt }\sim  \; \Gamma(z+1)(it)^{-z-1}+b(t),\end{align}where $b\in C^\infty[-\pi , \pi]$.  	Here the  notation $x\sim y$ stands for, there exists constants $C_1,~C_2>0$ such that $C_1 x\leq y\leq C_2x.$ 
	\end{lemma}
	When $\operatorname{Re}(z)=0$, as (\ref{CH2twoo}), we have
	\begin{align}\label{CH2two} \nonumber
		\|T_{is} \|_{L^{2}\left((-\frac{\pi}{2},\frac{\pi}{2}), L^2_{k,a}(\mathbb{R}^n)\right) \rightarrow  L^{2}\left((-\frac{\pi}{2},\frac{\pi}{2}), L^2_{k,a}(\mathbb{R}^n)\right) }& =\left\|G_{i s}\right\|_{\ell^{\infty}\left(\mathbb{Z}\times \mathcal{A} \right)} \\&\leq |\psi(is)|= \left|  \frac{1}{\Gamma(1+i s)} \right|\leq C e^{\pi|s| / 2}.	\end{align}
	When $z=-\lambda_{0}+is$, (\ref{CH21inf}) gives $T_z$  is bounded from  	$L^{1}\left((-\frac{\pi}{2},\frac{\pi}{2}), L^1_{k,a}(\mathbb{R}^n)\right)$ to $ L^{\infty}\left((-\frac{\pi}{2},\frac{\pi}{2}), L^\infty_{k,a}(\mathbb{R}^n)\right) $  if   $|K_z(t,x,y)|$ is uniformly bounded for each  $x, y\in\mathbb{R}^n$ and $t\in (-\frac{\pi}{2},\frac{\pi}{2})$. 
	When $a=1,2$, from  (\ref{CH2kernl}), Proposition \ref{CH2515}, and  Proposition \ref{eq200}, we get
	\begin{align}\label{CH2kernel}
		|K_z(t,x,y)|&\sim\frac{1}{ |t|^{\operatorname{Re}(z+1+\frac{2\gamma+n+a-2}{a})}}, ~\forall x,y\in\mathbb{R}^n, t\in(-\frac{\pi}{2},\frac{\pi}{2}).  
	\end{align}
	Therefore, for every $x,y\in \mathbb{R}$ and $t\in(-\frac{\pi}{2},\frac{\pi}{2})$,  $|K_z(t,x,y)|$ is uniformly bounded if $\operatorname{Re}(z)=-1-\frac{2\gamma+n+a-2}{a}.        $   Thus choosing $\lambda_0=1+\frac{2\gamma+n+a-2}{a}$, $T_z$  is bounded from  	$L^{1}\left((-\frac{\pi}{2},\frac{\pi}{2}), L^1_{k,a}(\mathbb{R}^n)\right)$ to $ L^{\infty}\left((-\frac{\pi}{2},\frac{\pi}{2}), L^\infty_{k,a}(\mathbb{R}^n)\right) $.
	
	Now the Schatten boundedness comes out from Lemma \ref{CH2diagggg}, 
	\begin{lemma}\label{as}  
		Suppose $a=1,2$ and $k$ is a non-negative multiplicity function such that
		$$
		a+2 \gamma+n-2>0.
		$$
		Let $n \geq 1$ and $S=\{(\nu, \ell, m , j)\in \mathbb{Z}\times \mathcal{A}: \nu=2\ell+\frac{2m}{a}\}   $ be a discrete  surface  on $\mathbb{Z}\times \mathcal{A}$ and $\lambda_0=1+\frac{2\gamma+n+a-2}{a}$.   Then  for all $W_1, W_2 \in L^{2\lambda_0}\left((-\frac{\pi}{2},\frac{\pi}{2}), L^{2\lambda_0}_{k,a}(\mathbb{R}^n)\right) ,$ the operator $W_{1} \mathcal T_{S} W_{2}=W_{1}   T_{-1} W_{2}$
		belongs to $\mathcal{G}^{2 \lambda_{0}} (L^{2} ((-\frac{\pi}{2},\frac{\pi}{2}), L^2_{k,a}(\mathbb{R}^n) )    )$ and we have the estimate $$	\left\|W_{1} \mathcal T_{S} W_{2}\right\|_{\mathcal{G}^{2 \lambda_{0}}\left(L^{2}\left((-\frac{\pi}{2},\frac{\pi}{2}), L^2_{k,a}(\mathbb{R}^n)\right)  \right)} \leq  C \left\|W_{1}\right\| _{L^{2\lambda_0}\left((-\frac{\pi}{2},\frac{\pi}{2}), L^{2\lambda_0}_{k,a}(\mathbb{R}^n)\right)  } \left\|W_{2 }\right\| _{L^{2\lambda_0}\left((-\frac{\pi}{2},\frac{\pi}{2}), L^{2\lambda_0}_{k,a}(\mathbb{R}^n)\right)  }.$$
	\end{lemma}
	
	\section{Restriction theorems and Strichartz inequalities for orthonomal functions}\label{s5}
	Let $S=\{(\nu, \ell, m , j)\in \mathbb{Z}\times \mathcal{A}: \nu=2\ell+\frac{2m}{a}\}   $ be a discrete  surface  on $\mathbb{Z}\times \mathcal{A}$ and $\lambda_0=1+\frac{2\gamma+n+a-2}{a}$. From Theorem \ref{as}, we have
	\begin{align}
		&\left\|W_1\mathcal{E}_S\left(\mathcal{E}_S\right)^*W_2\right\|_{\mathcal{G}^{2 \lambda_{0}}\left(L^{2}\left((-\frac{\pi}{2},\frac{\pi}{2}), L^2_{k,a}(\mathbb{R}^n)\right)  \right)}\nonumber \\
		&\leq  C \left\|W_{1}\right\| _{L^{2\lambda_0}\left((-\frac{\pi}{2},\frac{\pi}{2}), L^{2\lambda_0}_{k,a}(\mathbb{R}^n)\right)  } \left\|W_{2 }\right\| _{L^{2\lambda_0}\left((-\frac{\pi}{2},\frac{\pi}{2}), L^{2\lambda_0}_{k,a}(\mathbb{R}^n)\right)  }.\label{Schatten-decompostion}
	\end{align} 
	Taking $A=\mathcal{E}_S$, by Theorem \ref{Restriction-single}, Lemma \ref{duality-principle}, and \eqref{Schatten-decompostion} , we have the restriction theorem for the system of orthonomal functions.
	\begin{theorem}[Restriction estimates for orthonormal functions-diagonal case] \label{diagonal-restriction} Suppose $a=1,2$ and $k$ is a non-negative multiplicity function such that
		$$
		a+2 \gamma+n-2>0.
		$$
		Let $n\geq 1$ and $S=\{(\nu, \ell, m , j)\in \mathbb{Z}\times \mathcal{A}: \nu=2\ell+\frac{2m}{a}\}$.
		For any (possible infinity) orthonormal system $\left\{\{\hat{F_\iota}(\nu, \ell, m , j)\}_{(\nu, \ell, m , j)\in \mathbb{Z}\times \mathcal{A}}\right\}_{\iota \in{I}}$ in $\ell^2(S)$ and any sequence $\{n_\iota\}_{\iota \in{I}}$ in $\mathbb{C}$
		\begin{equation*}
			\begin{aligned}
				\bigg\|\sum_{\iota\in I}n_\iota|\mathcal{E}_S\{\{\hat{F_\iota}(\nu, \ell, m , j)\}|^2\bigg\|_{L^{1+\frac{a}{2\gamma+n+a-2}}\left((-\frac{\pi}{2},\frac{\pi}{2}), L^{1+\frac{a}{2\gamma+n+a-2}}_{k,a}(\mathbb{R}^n)\right)}
				\leq C\left\|\{n_\iota\}_{\iota\in{I}}\right\|_{\ell^{\frac{2(2\gamma+n+2a-2)}{4\gamma+2n+3a-4}}},
			\end{aligned}
		\end{equation*}		
		where $C>0$ is independent of $\left\{\{\hat{F_\iota}(\nu, \ell, m , j)\}_{(\mu,\nu)\in\mathbb{N}^n\times\mathbb{Z}}\right\}_{\iota\in{I}}$ and $\{n_\iota\}_{\iota\in{I}}$.
	\end{theorem}
	By \eqref{eq11}, we generalize the Strichartz inequality involving systems of orthonormal functions.
	\begin{theorem}[Strichartz inequalities for orthonormal functions-diagonal case]\label{diagonal} Suppose $a=1,2$, $n\geq1$ and $k$ is a non-negative multiplicity function such that
		$$
		a+2 \gamma+n-2>0.
		$$
		For any (possible infinity) orthonormal system $\left\{f_\iota\right\}_{\iota\in{I}}$ in $L_{k,a}^2(\mathbb{R}^n)$ and any sequence $\{n_\iota\}_{\iota\in{I}}$ in $\mathbb{C}$, we have
		\begin{equation*}
			\begin{aligned}
				\bigg\|\sum_{\iota \in I}n_\iota|e^{-it\Delta_ {k,a}}f_\iota|^2\bigg\|_{L^{1+\frac{a}{2\gamma+n+a-2}}\left((-\frac{\pi}{2},\frac{\pi}{2}), L^{1+\frac{a}{2\gamma+n+a-2}}_{k,a}(\mathbb{R}^n)\right)}\leq C\left\|\{n_\iota\}_{\iota\in{I}}\right\|_{\ell^{\frac{2(2\gamma+n+2a-2)}{4\gamma+2n+3a-4}}},
			\end{aligned}
		\end{equation*}		
		where $C>0$ is independent of $\{f_\iota \}_{\iota\in I}$ and $\{n_\iota \}_{\iota \in I }$.
	\end{theorem}
	To obtain the Strichartz inequality for the general case, we have to  prove the following Schatten boundedness.
	\begin{proposition}
		[General Schatten bound for the extension operator]\label{general-schatten-bound}		Suppose $a=1,2$ and $k$ is a non-negative multiplicity function such that
		$$
		a+2 \gamma+n-2>0.
		$$
		Let $n\geq 1$ and $S=\{(\nu, \ell, m , j)\in \mathbb{Z}\times \mathcal{A}: \nu=2\ell+\frac{2m}{a}\}$.
		Then for any $p,q\geq 1$ satisfying 
		\begin{equation*}
			\frac{1}{q}+\frac{2 \gamma+n+a-2}{pa}=\frac{1}{2},\quad \frac{4\gamma+2n+3a-4}{a}<p\leq\frac{2(2\gamma+n+2a-2)}{a} ,
		\end{equation*}
		we have
		\begin{equation*}
			\left\|W_1 \mathcal{T}_SW_2\right\|_{\mathcal{G}^p\left(L^2((-\frac{\pi}{2},\frac{\pi}{2}), L_{k,a}^2(\mathbb{R}^n))\right)}\leq C\left\|W_1\right\|_{L^q((-\frac{\pi}{2},\frac{\pi}{2}), L_{k,a}^p(\mathbb{R}^n))}\left\|W_2\right\|_{L^q((-\frac{\pi}{2},\frac{\pi}{2}), L_{k,a}^p(\mathbb{R}^n))},
		\end{equation*} 
		for all $W_1,W_2\in L^q((-\frac{\pi}{2},\frac{\pi}{2}), L_{k,a}^p(\mathbb{R}^n))$, where $C>0$ is independent of $W_1,W_2$.
	\end{proposition}
	In order to prove Proposition \ref{general-schatten-bound}, we need to obtain the following revised Hardy-Littlewood-Sobolev inequality.
	\begin{lemma}\label{r-HLS}
		For $0\leq \lambda<1$ and $1<p,q<\infty$ such that $\frac{1}{p}+\frac{1}{q}+\lambda=2$, we have
		\begin{equation*}
			\left|\int^{\frac{\pi}{2}}_{-\frac{\pi}{2}}\int^{\frac{\pi}{2}}_{-\frac{\pi}{2}}   \frac{g(t)h(\tau)}{|\sin(t-\tau)|^\lambda}dtd\tau\right|\leq C\|g\|_{L^p(-\frac{\pi}{2},\frac{\pi}{2})}\|h\|_{L^q(-\frac{\pi}{2},\frac{\pi}{2})}.
		\end{equation*}
	\end{lemma}
	\begin{proof}
		Without loss of generality, we assume $g,h\geq 0$.
		We divide $(-\frac{\pi}{2},\frac{\pi}{2})^2$ into three disjoint subsets:
		$$B_1=\left \{(t,\tau)\in (-\frac{\pi}{2},\frac{\pi}{2})^2: |t-\tau|\leq \frac{\pi}{2}\right\},$$ $$ B_2=\left\{(t,\tau)\in (-\frac{\pi}{2},\frac{\pi}{2})^2: \frac{\pi}{2}<t-\tau<\pi\right\},$$  and $$ B_3=\left\{(t,\tau)\in (-\frac{\pi}{2},\frac{\pi}{2})^2: \frac{\pi}{2}<\tau-t<\pi\right\}.$$ Then the integral is split into three parts
		\begin{equation*}
			I=\int^{\frac{\pi}{2}}_{-\frac{\pi}{2}}\int^{\frac{\pi}{2}}_{-\frac{\pi}{2}}   \frac{g(t)h(\tau)}{|\sin(t-\tau)|^\lambda}dtd\tau:=I_1+I_2+I_3,
		\end{equation*}
		where
		\begin{equation*}
			I_j=\int_{B_j}  \frac{g(t)h(\tau)}{|\sin(t-\tau)|^\lambda}dtd\tau,~\forall j=1,2,3.
		\end{equation*}
		For any $t\in [0,\frac{\pi}{2}]$, we have $\frac{2}{\pi}t\leq \sin t\leq t$.
		Thus when $(t,\tau)\in B_1$,   we get $\frac{2}{\pi}|t-\tau|\leq |\sin (t-\tau)|\leq |t-\tau|$. Then
		\begin{align}\label{I1}\nonumber
			I_1&=\int_{B_1}  \frac{g(t)h(\tau)}{|\sin(t-\tau)|^\lambda}dtd\tau\\\nonumber
			&\leq \left(\frac{\pi}{2}\right)^\lambda\int^{\frac{\pi}{2}}_{-\frac{\pi}{2}}\int^{\frac{\pi}{2}}_{-\frac{\pi}{2}}   \frac{g(t)h(\tau)}{|t-\tau|^\lambda}dtd\tau\\
			&\leq C\|g\|_{L^p(-\frac{\pi}{2},\frac{\pi}{2})}\|h\|_{L^q(-\frac{\pi}{2},\frac{\pi}{2})}.
		\end{align}
		When $(t,\tau)\in B_2$, we have $0<\pi+\tau-t<\frac{\pi}{2}$ and $(t,\pi+\tau)\in (-\frac{\pi}{2},\frac{\pi}{2})\times(\frac{\pi}{2},\frac{3\pi}{2})$. Then $\frac{2}{\pi}(\pi+\tau-t)\leq \sin (\pi+\tau-t)\leq (\pi+\tau-t)$, i.e., $\frac{2}{\pi}(\pi+\tau-t)\leq \sin (t-\tau)\leq (\pi+\tau-t)$. Thus it follows that
		\begin{align}\label{I2}\nonumber
			I_2&=\int_{B_2}  \frac{g(t)h(\tau)}{|\sin(t-\tau)|^\lambda}dtd\tau\\\nonumber
			&\leq \left(\frac{\pi}{2}\right)^\lambda\int_{B_2} \frac{g(t)h(\tau)}{|\pi+\tau-t|^\lambda}dtd\tau\\\nonumber
			&\leq\int^{\frac{\pi}{2}}_{-\frac{\pi}{2}}g(t)dt\int^{\frac{3\pi}{2}}_{\frac{\pi}{2}}\frac{h(s-\pi)}{|s-t|^\lambda}ds  \\\nonumber
			&\leq C\|g\|_{L^p(-\frac{\pi}{2},\frac{\pi}{2})}\|h(\cdot-\pi)\|_{L^q(\frac{\pi}{2},\frac{3\pi}{2})}\\
			&= C\|g\|_{L^p(-\frac{\pi}{2},\frac{\pi}{2})}\|h\|_{L^q(-\frac{\pi}{2},\frac{\pi}{2})}.
		\end{align}
		When $(t,\tau)\in B_3$, analogous to the subset $B_2$, we also  have \begin{align}\label{I3}
			I_3\leq C\|g\|_{L^p(-\frac{\pi}{2},\frac{\pi}{2})}\|h\|_{L^q(-\frac{\pi}{2},\frac{\pi}{2})}. \end{align}
		Therefore, from (\ref{I1}), (\ref{I2}), and (\ref{I3}), we have $$I=I_1+I_2+I_3\leq C\|g\|_{L^p(-\frac{\pi}{2},\frac{\pi}{2})}\|h\|_{L^q(-\frac{\pi}{2},\frac{\pi}{2})}.$$
	\end{proof}
	Now we are in a position  to prove Proposition \ref{general-schatten-bound}.
	\begin{proof}[Proof of Proposition \ref{general-schatten-bound}]
		For $0<\lambda<\lambda_0=1+\frac{2\gamma+n+a-2}{a}$, the operator $T_{-\lambda+is}$ is an integral operator with kernel $K_{-\lambda+is}(t-\tau,x,y)$ defined in (\ref{CH2kernl}), Applying Lemma \eqref{r-HLS} along with (\ref{CH2two}) and (\ref{CH2kernel}), we have  
		\begin{equation*}
			\begin{aligned}
				&\|W_1^{\lambda-is}T_{-\lambda+is}W_2^{\lambda-is}\|_{\mathcal{G}^2\left(L^2((-\frac{\pi}{2},\frac{\pi}{2}), L_{k,a}^2(\mathbb{R}^n))\right)}^2\\
				=&\int^{\frac{\pi}{2}}_{-\frac{\pi}{2}}\int^{\frac{\pi}{2}}_{-\frac{\pi}{2}}\int_{\mathbb{R}^{2n}}|W_1(t,x)|^{2\lambda}|K_{-\lambda+is}(t-\tau,x,y)|^2|W_2(\tau,y)|^{2\lambda}v_{k,a}(x)dx\;v_{k,a}(y)dydtd\tau\\
				\leq  &C\int^{\frac{\pi}{2}}_{-\frac{\pi}{2}}\int^{\frac{\pi}{2}}_{-\frac{\pi}{2}}\int_{\mathbb{R}^{2n}}\frac{|W_1(t,x)|^{2\lambda}|W_2(\tau,y)|^{2\lambda}}{|\sin(t-\tau)|^{2\lambda_0-2\lambda}}v_{k,a}(x)dx\;v_{k,a}(y)dydtd\tau\\
				\leq  &C\int^{\frac{\pi}{2}}_{-\frac{\pi}{2}}\int^{\frac{\pi}{2}}_{-\frac{\pi}{2}}\frac{\|W_1(t,\cdot)\|_{L_{k,a}^{2\lambda}(\mathbb{R}^n)}^{2\lambda}\|W_2(\tau,\cdot)\|_{L_{k,a}^{2\lambda}(\mathbb{R}^n)}^{2\lambda}}{|\sin(t-\tau)|^{2\lambda_0-2\lambda}}dtd\tau\\
				\leq  &C\|W_1\|_{L^{\frac{2\lambda}{1+\lambda-\lambda_0}}\left((-\frac{\pi}{2},\frac{\pi}{2}), L_{k,a}^{2\lambda}(\mathbb{R}^n)\right)}^{2\lambda}\|W_2\|_{L^{\frac{2\lambda}{1+\lambda-\lambda_0}}\left((-\frac{\pi}{2},\frac{\pi}{2}), L_{k,a}^{2\lambda}(\mathbb{R}^n)\right)}^{2\lambda},
			\end{aligned}
		\end{equation*}               
		provided $0\leq2\lambda_0-2\lambda<1$, i.e., $2\lambda_0-1<2\lambda\leq2\lambda_0$. By Theorem 2.9 of \cite{Simon} and the identity $\mathcal{T}_S=T_{-1}$, we have 
		\begin{equation}\label{part}
			\begin{aligned}
				&\left\|W_1 \mathcal{T}_SW_2\right\|_{\mathcal{G}^{2\lambda}\left(L^2((-\frac{\pi}{2},\frac{\pi}{2}), L_{k,a}^2(\mathbb{R}^n))\right)}=\left\|W_1 T_{-1}W_2\right\|_{\mathcal{G}^{2\lambda}\left(L^2((-\frac{\pi}{2},\frac{\pi}{2}), L_{k,a}^2(\mathbb{R}^n))\right)}\\  \leq &C\|W_1\|_{L^{\frac{2\lambda}{1+\lambda-\lambda_0}}\left((-\frac{\pi}{2},\frac{\pi}{2}), L_{k,a}^{2\lambda}(\mathbb{R}^n)\right)}\|W_2\|_{L^{\frac{2\lambda}{1+\lambda-\lambda_0}}\left((-\frac{\pi}{2},\frac{\pi}{2}), L_{k,a}^{2\lambda}(\mathbb{R}^n)\right).}
			\end{aligned} 
		\end{equation}
		Again	by Lemma \ref{duality-principle}, we have
		\begin{equation}\label{1}
			\begin{aligned}
				\bigg\|\sum_{\iota\in I}n_\iota|\mathcal{E}_S\{\{\hat{F_\iota}(\nu, \ell, m , j)\}|^2\bigg\|_{L^{\frac{\lambda}{\lambda_0-1}}\left((-\frac{\pi}{2},\frac{\pi}{2}), L^{\frac{\lambda}{\lambda-1}}(\mathbb{R}_+^n,dw_\alpha)\right)}
				\leq C\bigg(\sum_{\iota\in I}|n_\iota|^{(2\lambda)^\prime}\bigg)^{(2\lambda)^\prime},
			\end{aligned}
		\end{equation}		 
		for any orthonormal system $\left\{\{\hat{F_\iota}(\nu, \ell, m , j)\}\right\}_{\iota\in{I}}$ in $\ell^2 (S)$ and $\{n_\iota\}_{\iota\in{I}}$ in $\mathbb{C}$, which is equivalent to
		\begin{equation}\label{2}
			\begin{aligned}
				\bigg\|\sum_{\iota\in I}n_\iota|e^{-it\Delta_{k,a}}f_i|^2\bigg\|_{L^{\frac{\lambda}{\lambda_0-1}}\left((-\frac{\pi}{2},\frac{\pi}{2}), L_{k,a}^{\frac{\lambda}{\lambda-1}}(\mathbb{R}^n)\right)}
				\leq C\bigg(\sum_{\iota\in I}|n_\iota|^{(2\lambda)^\prime}\bigg)^{(2\lambda)^\prime},
			\end{aligned}
		\end{equation}		
		for any orthonormal system $\{f_\iota\}_{\iota\in I}$ in $L_{k,a}^2(\mathbb{R}^n)$ and $\{n_\iota\}_{\iota\in I}$ in $\mathbb{C}$. This completes the proof of the proposition.
	\end{proof}
	Now we are ready to establish the restriction theorem for systems of orthonormal functions.
	\begin{theorem}[Restriction estimates for orthonormal functions-general case] 	Suppose $a=1,2$ and $k$ is a non-negative multiplicity function such that
		$
		a+2 \gamma+n-2>0.
		$ Let $n\geq 1$ and $S=\{(\nu, \ell, m , j)\in \mathbb{Z}\times \mathcal{A}: \nu=2\ell+\frac{2m}{a}\}$. If $p, q, n \geqslant$ 1 such that
		$$
		1 \leqslant p<\frac{4\gamma+2n+3a-4}{  4\gamma+2n+a-4} \quad \text { and } \quad 	\frac{1}{q}+\frac{2 \gamma+n+a-2}{pa}=\frac{2 \gamma+n+a-2}{a},
		$$
		for any (possible infinity) orthonormal system $\left\{\{\hat{F_\iota}(\nu, \ell, m , j)\}_{(\nu, \ell, m , j)\in \mathbb{Z}\times \mathcal{A}}\right\}_{\iota \in{I}}$  in $\ell^{2}(S)$ and any sequence $\left\{n_{\iota }\right\}_{\iota \in I}$ in $\mathbb{C}$, we have
		\begin{align}\label{p-q range}
			\left \| \sum_{\iota \in I} n_{\iota } \left|  \mathcal{E}_{S} \left\{\hat{F}_{\iota }(\nu, \ell, m , j)\right\} \right|^{2} \right\|_{L^q((-\frac{\pi}{2},\frac{\pi}{2}), L_{k,a}^p(\mathbb{R}^n))} \leqslant C\left(\sum_{\iota }\left|n_{\iota }\right|^{\frac{2 p}{p+1}}\right)^{\frac{p+1}{2 p}},
		\end{align}
		where $C>0$ is independent of  $\left\{\{\hat{F_\iota}(\nu, \ell, m , j)\}_{(\nu, \ell, m , j)\in \mathbb{Z}\times \mathcal{A}}\right\}_{\iota \in{I}}$  and  $\left\{n_{\iota}\right\}_{\iota  \in J}$.
		
	\end{theorem}
	
	\begin{proof}
		Using the fact that the operator $e^{-i t \Delta_{k, a}}$ is unitary, 	 we have
		$$
		\left\|\sum_{i \in \iota } n_{\iota }\left|e^{-i t \Delta_{k, a}}f_{\iota }\right|^{2}\right\|_{L^\infty ((-\frac{\pi}{2},\frac{\pi}{2}), L_{k,a}^1(\mathbb{R}^n))} \leqslant \sum_{\iota \in I}\left|n_{\iota }\right|,
		$$
		for any (possible infinity) system $\left\{f_{\iota}\right\}_{\iota \in I }$ of orthonormal functions in $L^{2}\left(\mathbb{R}_{+}^{n}, d w_{\alpha}\right)$ and any coefficients $\left\{n_{\iota}\right\}_{\iota \in I }$ in $\mathbb{C}$. It is equivalent to
		$$
		\left \| \sum_{\iota \in I} n_{\iota } \left|  \mathcal{E}_{S} \left\{\hat{F}_{\iota }(\nu, \ell, m , j)\right\} \right|^{2} \right\|_{L^\infty ((-\frac{\pi}{2},\frac{\pi}{2}), L_{k,a}^1(\mathbb{R}^n))} \leqslant \sum_{\iota \in I}\left|n_{\iota}\right|, 
		$$
		for any (possible infinity) orthonormal system $\left\{\{\hat{F_\iota}(\nu, \ell, m , j)\}_{(\nu, \ell, m , j)\in \mathbb{Z}\times \mathcal{A}}\right\}_{\iota \in{I}}$ in $\ell^{2}(S)$ and any sequence $\left\{n_{\iota}\right\}_{\iota \in I }$ in $\mathbb{C}$.
		
		By Lemma \ref{duality-principle}, equivalently, we have the operator
		\begin{equation*}
			W\in L^2((-\frac{\pi}{2},\frac{\pi}{2}), L_{k,a}^\infty(\mathbb{R}^n))\longmapsto W\mathcal{T}_S\overline{W}\in \mathcal{G}^\infty\left(L^2((-\frac{\pi}{2},\frac{\pi}{2}), L_{k,a}^2(\mathbb{R}^n))\right)
		\end{equation*}
		is bounded. By Lemma \ref{as}, the operator 
		\begin{align*}
			&W\in L^{\frac{2(2\gamma+n+2a-2)}{a}}((-\frac{\pi}{2},\frac{\pi}{2}), L_{k,a}^{\frac{2(2\gamma+n+2a-2)}{a}}(\mathbb{R}^n))\\
			&\longmapsto W\mathcal{T}_S\overline{W}\in \mathcal{G}^{2(1+\frac{2\gamma+n+a-2}{a})}\left(L^2((-\frac{\pi}{2},\frac{\pi}{2}), L_{k,a}^2(\mathbb{R}^n))\right)
		\end{align*}
		is also bounded. Applying the complex interpolation method in the Chapter 4 of \cite{BL}, the operator
		\begin{equation*}
			W\in L^q((-\frac{\pi}{2},\frac{\pi}{2}), L_{k,a}^p(\mathbb{R}^n))\longmapsto W\mathcal{T}_S\overline{W}\in \mathcal{G}^p\left(L^2((-\frac{\pi}{2},\frac{\pi}{2}), L_{k,a}^2(\mathbb{R}^n,dw))\right)
		\end{equation*}
		is bounded for any $p,q\ge1$ satisfying
		\begin{equation*}
			\frac{1}{q}+\frac{2 \gamma+n+a-2}{pa}=\frac{1}{2},\quad \frac{2(2\gamma+n+2a-2)}{a}\leq p\leq\infty.
		\end{equation*}		
		Combined with Lemma \ref{general-schatten-bound}, we have the operator
		\begin{equation*}
			W\in L^q((-\frac{\pi}{2},\frac{\pi}{2}), L_{k,a}^p(\mathbb{R}^n))\longmapsto W\mathcal{T}_S\overline{W}\in \mathcal{G}^p\left(L^2((-\frac{\pi}{2},\frac{\pi}{2}), L_{k,a}^2(\mathbb{R}^n,dw))\right)
		\end{equation*} is still bounded for any $p,q\ge1$ satisfying
		\begin{equation*}
			\frac{1}{q}+\frac{2 \gamma+n+a-2}{pa}=\frac{1}{2},\quad \frac{4\gamma+2n+3a-4}{a}<p\leq\infty.
		\end{equation*}
		Again by Theorem \ref{Restriction-single} and Lemma \ref{duality-principle}, 
		we obtain our desired estimates.
	\end{proof}
	
	As a consequences of the above result, we   obtain the following Strichartz inequalities for the system of orthonormal functions.
	
	\begin{theorem} [Strichartz inequalities for orthonormal functions-general case]\label{strichartz-general} If $p, q, n \geqslant$ 1 such that
		$$
		1 \leqslant p<\frac{4\gamma+2n+3a-4}{  4\gamma+2n+a-4} \quad \text { and } \quad 	\frac{1}{q}+\frac{2 \gamma+n+a-2}{pa}=\frac{2 \gamma+n+a-2}{a},
		$$
		then for any (possible infinity) system $\left\{f_{\iota}\right\}_{\iota \in I }$ of orthonormal functions in $L_{k,a}^{2}\left(\mathbb{R}^{n}\right)$ and any coefficients $\left\{n_{\iota}\right\}_{\iota \in I }$ in $\mathbb{C}$, we have
		$$
		\left\|\sum_{  \iota \in I} n_{\iota }\left|e^{-i t \Delta_{k, a}}f_{\iota }\right|^{2}\right\|_{L^q ((-\frac{\pi}{2},\frac{\pi}{2}), L_{k,a}^p(\mathbb{R}^n))} \leqslant C\left(\sum_{\iota \in I}\left|n_{\iota }\right|^{\frac{2 p}{p+1}}\right)^{\frac{p+1}{2 p}},
		$$
		where $C>0$ is independent of $\left\{f_{\iota}\right\}_{\iota\in I}$ and $\left\{n_{\iota}\right\}_{\iota \in I}$.
	\end{theorem}
	
	\section{Orthonormal  Strichartz inequalities for the  Dunkl operator}\label{s6}
	In this section we prove   Strichartz estimate for system of orthonormal functions associated with    Dunkl operator $\Delta_{k} $  using  the relation between the   semigroups corresponding to    $(k, a)$-generalized Laguerre operator  $\Delta_{k, a} $ and  the Dunkl operator $\Delta_{k} $.
	
	Corresponding to the Dunkl operator $\Delta_{k} $, we have the solution	$e^{i \frac{t}{a}\|x\|^{2-a} \Delta_{k}} f$   to the free Schr\"odinger equation 
	$$ \begin{cases}
		i \partial_{t} w(t, x)+\frac{1}{a}\|x\|^{2-a} \Delta_{k} w(t, x)=0, & (t, x) \in \mathbb{R} \times\mathbb{R}^n, \\
		w(0,x)=f(x), & x \in \mathbb{R}^n.
	\end{cases} $$
	For $a=1,2$ and $k$ is a non-negative multiplicity function such that
	$
	a+2 \gamma+n-2>0
	$, from  Theorem \ref{111}, the classical estimates in this case    may be stated as  
	\begin{equation*}
		\|e^{i \frac{t}{a}\|x\|^{2-a} \Delta_{k}} f\|_{L^{q}\left( \mathbb{R},L_{k,a}^p(\mathbb{R}^n)\right)}\leq C\|f\|_{L_{k, a}^{2}\left(\mathbb{R}^{n}\right)}
	\end{equation*}
	under the same conditions on $p$ and $q$ is that 
	$1\leq p, q\leq\infty$ satisfies $$\left(\frac{1}{2}-\frac{1}{p}\right)\frac{2 \gamma+n+a-2}{a}-\frac{1}{q}=0.$$
	As a consequences of the above result, we   obtain the following Strichartz inequality associated with Dunkl operator  for the system of orthonormal functions.
	
	\begin{theorem} [Strichartz inequalities for orthonormal functions associated with Dunkl operator-general case]\label{OSI-Dunkl} If $p, q, n \geqslant$ 1 such that
		$$
		1 \leqslant p<\frac{4\gamma+2n+3a-4}{  4\gamma+2n+a-4} \quad \text { and } \quad 	\frac{1}{q}+\frac{2 \gamma+n+a-2}{pa}=\frac{2 \gamma+n+a-2}{a},
		$$
		then for any (possible infinity) system $\left\{f_{\iota}\right\}_{\iota \in I }$ of orthonormal functions in $L_{k,a}^{2}\left(\mathbb{R}^{n}\right)$ and any coefficients $\left\{n_{\iota}\right\}_{\iota \in I }$ in $\mathbb{C}$, we have
		$$
		\left\|\sum_{\iota\in I} n_{\iota}\left|e^{i\frac{ t}{a} \|x\|^{2-a}\Delta_k} f_{\iota}\right|^{2} \right\|_{L^q (\mathbb{R}, L_{k,a}^p(\mathbb{R}^n))} \leqslant C\left(\sum_{\iota \in I}\left|n_{\iota }\right|^{\frac{2 p}{p+1}}\right)^{\frac{p+1}{2 p}},
		$$
		where $C>0$ is independent of $\left\{f_{\iota}\right\}_{\iota\in I}$ and $\left\{n_{\iota}\right\}_{\iota \in I}$.
	\end{theorem}
	\begin{proof}
		Let  $\Gamma_{k, a}(x, y ; i t)$  and $\Lambda_{k, a}(x, y ; i t) $ be the kernel of $e^{i \frac{t}{a}\|x\|^{2-a} \Delta_{k}}$ and $e^{-i t\Delta_{k, a}}$, respectively. Then we know that, using the change of variable $s=\tan (t)$ with $t \in(-\pi / 2, \pi / 2),$ we get
		\begin{align*}
			\Lambda_{k, a}(x, y ; i \tan^{-1} s)&=c_{k, a}^{-1}\left(1+s^{2}\right)^{\frac{2 y+n+a-2}{2 a}} \exp \left(-i s \frac{\|x\|^{a}}{a}\right) \Gamma_{k, a}\left(\left(1+s^{2}\right)^{\frac{1}{a}} x, y ; i s\right) 
		\end{align*}
		for any $s>0$.  From this, and using the scaling condition  $\frac{1}{q}+\frac{2 \gamma+n+a-2}{pa}=\frac{2 \gamma+n+a-2}{a}$, it follows (as in Theorem \ref{111} for a single function) that 
		\begin{align*}
			\left\|\sum_{\iota\in I} n_{\iota}\left|e^{i\frac{ t}{a} \|x\|^{2-a}\Delta_k} f_{\iota}\right|^{2} \right\|_{L^{q}\left( ( 0, +\infty),L_{k,a}^p(\mathbb{R}^n)\right)}&=\left\|\sum_{  \iota \in I} n_{\iota }\left|e^{-i t \Delta_{k, a}}f_{\iota }\right|^{2}\right\|_{L^{q}\left( ( 0, \frac{\pi }{2}),L_{k,a}^p(\mathbb{R}^n)\right)},
		\end{align*}
		and
		\begin{align*}
			\left\|\sum_{\iota\in I} n_{\iota}\left|e^{i\frac{ t}{a} \|x\|^{2-a}\Delta_k} f_{\iota}\right|^{2} \right\|_{L^{q}\left( (-\infty,0),L_{k,a}^p(\mathbb{R}^n)\right)}&=\left\|\sum_{  \iota \in I} n_{\iota }\left|e^{-i t \Delta_{k, a}}f_{\iota }\right|^{2}\right\|_{L^{q}\left( (-\frac{\pi }{2},0),L_{k,a}^p(\mathbb{R}^n)\right)},
		\end{align*}
		for any system of orthonormal functions $\left\{f_{\iota}\right\}_{\iota\in I}$ in $L_{k,a}^2(\mathbb{R}^n)$ and any coefficients $\left\{n_{\iota}\right\}_{\iota \in I}$ in $\mathbb{C}$. Hence, we immediately obtain our desired results from Theorem \ref{strichartz-general}.
	\end{proof}
	
	\section{Final remarks}\label{s7}
	\begin{enumerate}
		\item When $a=2$ and $k\in\mathcal{K}^+$, then Theorem \ref{OSI-Dunkl} becomes: if
		$$
		p, q, n \geqslant 1,\quad 1 \leqslant p<\frac{2\gamma+n+1}{  2\gamma+n-1} \quad \text { and } \quad 	\frac{2}{q}+\frac{2 \gamma+n}{p}=2 \gamma+n,
		$$
		for any (possible infinity) system $\left\{f_{\iota}\right\}_{\iota \in I }$ of orthonormal functions in $L_{k}^{2}\left(\mathbb{R}^{n}\right)$ and any coefficients $\left\{n_{\iota}\right\}_{\iota \in I }$ in $\mathbb{C}$, we have
		$$
		\left\|\sum_{\iota\in I} n_{\iota}\left|e^{it\Delta_k} f_{\iota}\right|^{2} \right\|_{L^q (\mathbb{R}, L_{k}^p(\mathbb{R}^n))} \leqslant C\left(\sum_{\iota \in I}\left|n_{\iota }\right|^{\frac{2 p}{p+1}}\right)^{\frac{p+1}{2 p}},
		$$
		which generalizes the known Strichartz estimates for the Schr\"odinger-Dunkl equation in \cite{Majjaoli}  involving systems of orthonormal functions.	 
		
		Moreover, when $a=2$ and $k \equiv 0$, then Theorem \ref{OSI-Dunkl} reduces: if
		$$
		p, q, n \geqslant 1,\quad	1 \leqslant p<\frac{n+1}{  n-1} \quad \text { and } \quad 	\frac{2}{q}+\frac{n}{p}={n},
		$$
		for any (possible infinity) system $\left\{f_{\iota}\right\}_{\iota \in I }$ of orthonormal functions in $L^{2}\left(\mathbb{R}^{n}\right)$ and any coefficients $\left\{n_{\iota}\right\}_{\iota \in I }$ in $\mathbb{C}$, we have
		$$
		\left\|\sum_{\iota\in I} n_{\iota}\left|e^{it \Delta} f_{\iota}\right|^{2} \right\|_{L^q (\mathbb{R}, L^p(\mathbb{R}^n))} \leqslant C\left(\sum_{\iota \in I}\left|n_{\iota }\right|^{\frac{2 p}{p+1}}\right)^{\frac{p+1}{2 p}},
		$$
		which is nothing but the classical orthonormal Strichartz inequality associated with Laplacian $\Delta$ on $\mathbb{R}^n$ proved  in \cite{frank, FS}.
		
		\item When $a=2$ and $k\in\mathcal{K}^+$, then Theorem \ref{strichartz-general} becomes the orthonormal Strichartz inequality for the Schr\"odinger equation associated with the Dunkl-Hermite operator $H_k$ which states: if
		
		$$p, q, n \geqslant 1,\quad 
		1 \leqslant p<\frac{2\gamma+n+1}{  2\gamma+n-1} \quad \text { and } \quad 	\frac{2}{q}+\frac{2 \gamma+n}{p}=2 \gamma+n,
		$$
		for any (possible infinity) system $\left\{f_{\iota}\right\}_{\iota \in I }$ of orthonormal functions in $L_{k}^{2}\left(\mathbb{R}^{n}\right)$ and any coefficients $\left\{n_{\iota}\right\}_{\iota \in I }$ in $\mathbb{C}$, we have
		$$
		\left\|\sum_{  \iota \in I} n_{\iota }\left|e^{-i t H_k} f_{\iota }\right|^{2}\right\|_{L^q ((-\frac{\pi}{2},\frac{\pi}{2}), L_{k}^p(\mathbb{R}^n))} \leqslant C\left(\sum_{\iota \in I}\left|n_{\iota }\right|^{\frac{2 p}{p+1}}\right)^{\frac{p+1}{2 p}},
		$$
		which generalizes the known Strichartz estimates for the Dunkl-Hermite operator in Theorem 4.5 of \cite{Majjaoli}.
		
		Moreover, when $a=2$ and $k \equiv 0$, then Theorem \ref{strichartz-general} reduces to: for
		$$
		1 \leqslant p<\frac{n+1}{  n-1} \quad \text { and } \quad 	\frac{2}{q}+\frac{n}{p}={n}
		$$
		and any (possible infinity) system $\left\{f_{\iota}\right\}_{\iota \in I }$ of orthonormal functions in $L^{2}\left(\mathbb{R}^{n}\right)$ and any coefficients $\left\{n_{\iota}\right\}_{\iota \in I }$ in $\mathbb{C}$, we have
		$$
		\left\|\sum_{\iota\in I} n_{\iota}\left|e^{-it H} f_{\iota}\right|^{2} \right\|_{L^q ((-\frac{\pi}{2},\frac{\pi}{2}), L^p(\mathbb{R}^n))} \leqslant C\left(\sum_{\iota \in I}\left|n_{\iota }\right|^{\frac{2 p}{p+1}}\right)^{\frac{p+1}{2 p}},
		$$
		which is nothing but the classic Strichartz inequality for systems of  orthonormal functions associated with the Hermite operator  $H$ on $\mathbb{R}^n$ proved    in \cite{shyam, lee}.
		
		\item When $a=1$ and $k\in\mathcal{K}^+$, then Theorem \ref{strichartz-general} generalizes the Strichartz inequality for the Schr\"odinger propagator associated with the Dunkl-Laguerre operator $L_k$ on $\mathbb{R}^n$ (see Theorem A in \cite{Ben said}) involving systems of orthonormal functions and it states that: 
		If $p, q, n \geqslant$ 1 such that
		$$
		1 \leqslant p<\frac{ 4\gamma+2n-1}{4\gamma+2n-3} \quad \text { and } \quad 	\frac{1}{q}+\frac{ 2\gamma+n-1}{p}= 2\gamma+n-1,
		$$
		we have
		$$
		\left\|\sum_{\iota\in I} n_{\iota}\left|e^{-it L_k} f_{\iota}\right|^{2} \right\|_{L^q ((-\frac{\pi}{2},\frac{\pi}{2}), L_{k,1}^p(\mathbb{R}^n))} \leqslant C\left(\sum_{\iota \in I}\left|n_{\iota }\right|^{\frac{2 p}{p+1}}\right)^{\frac{p+1}{2 p}},
		$$ 
		for any (possible infinity) system $\left\{f_{\iota}\right\}_{\iota \in I }$ of orthonormal functions in $L_{k,1}^{2}\left(\mathbb{R}^{n}\right)$ and any coefficients $\left\{n_{\iota}\right\}_{\iota \in I }$ in $\mathbb{C}$. Besides, it reduces the orthonormal Strichartz inequality for the propagator $e^{it\|x\|\Delta_k}$ which is a generalization of Theorem D in \cite{Ben said}. Moreover, when $k\equiv0$, we can also obtain the orthonormal Strichartz inequality associated with the classical Laguerre operator $L$. We shall not list any more. We just intend to show the  generality of our approach. 
	\end{enumerate}

	\section*{Acknowledgments}
The first     author is supported by Core Research Grant (RP03890G), Science and Engineering Research Board (SERB), DST, India.	He also wishes to thank Indian Institute of Technology Delhi. The second author is supported by the National Natural Science Foundation of China (Grant No. 11701452) and the Natural Science Foundation of Shaanxi Province (Grant No. 2020JQ-112).


\begin{thebibliography}{99}
		
		
		\normalsize
		\baselineskip=17pt
		
		
		
		
		
		
		
		
		
		\bibitem{Ben said} 	 S. Ben Sa\"id, Strichartz estimates for Schr\"odinger-Laguerre operators, \emph{Semigroup Forum} 90, 251-269 (2015). 
		
		
		
		\bibitem{ben}S. Ben Said, T. Kobayashi and B. Orsted, Laguerre semigroup and Dunkl operators, \emph{Compos. Math.} 148, 1265-1336 (2012). 
		
		\bibitem{Ratna3}  S. Ben Said, A. K. Nandakumaran and P. K. Ratnakumar, Schr\"odinger propagator and the Dunkl Laplacian,   \emph{hal-00578446v1} (2011).
		
		\bibitem{BL} J. Bergh and J. L\"ofstr\"om, Interpolation Spaces. An Introduciton, Grundlehren Math. Wiss., 223, Springer (1976).
		
		\bibitem{lee} N. Bez, Y. Hong, S. Lee, S.  Nakamura and Y. Sawano,    {On the Strichartz estimates for orthonormal systems of initial data with regularity}, \emph{Adv. Math.}   {354}, Paper No. 106736 (2019).
		
		
		
		\bibitem{BEZ} N.  Bez, Neal, S.  Lee and S.  Nakamura,  Strichartz estimates for orthonormal families of initial data and weighted oscillatory integral estimates,  \emph{Forum Math. Sigma} 9, Paper No. e1  (2021).
		
		\bibitem{dun} C. F. Dunkl, Differential-difference operators associated to reflection groups, \emph{Trans. Amer. Math. Soc.} 311, 167-183 (1989).
		
		\bibitem{dun1991} C. F. Dunkl, Integral kernels with reflection group invariance, \emph{Canad. J. Math.} 43, 1213-1227 (1991).
		
		
		
		\bibitem{manli} G.  Feng and M.  Song, Restriction theorems and Strichartz inequalities for the Laguerre operator involving orthonormal functions,   	arXiv:2205.10360 (2022).
		
		\bibitem{frank}  R. L. Frank, M. Lewin,  E. H. Lieb and   R. Seiringer, Strichartz inequality for orthonormal functions, \emph{J. Eur. Math. Soc. (JEMS)} 16(7), 1507-1526 (2014).
		
		
		\bibitem{FS} R. L. Frank and J. Sabin, Restriction theorems for orthonormal functions, Strichartz inequalities, and uniform Sobolev estimates, \emph{Amer. J. Math.} 139(6), 1649-1691 (2017).
		
		
 
		
		
		\bibitem{sh} I. M. Gelfand and G. E. Shilov,  \textit{Generalized Functions}. vol. 1, Properties and operations, AMS Chelsea Publishing, New York (1964).
		
		\bibitem{hum} 	J. E. Humphreys, \emph{Reflection Groups and Coxeter Groups}, Cambridge Stud. Adv. Math. 29, Cambridge University Press, Cambridge (1990).
		
		\bibitem{lieb}  E. H. Lieb and   W. E. Thrring, Bound on kinetic energy of fermions which proves stability of matter,  \emph{Phys. Rev. Lett.} 35, 687-689 (1975).
		
		\bibitem{liebb}  E. H. Lieb and   W. E. Thrring, \emph{Inequalities for the moments of the eigenvalues of the Schr\"odinger hamiltonian and their relation to Sobolev inequalities}, Studies in Mathematical Physics, Princeton university press, page 269-303, Princeton (1976).
		
		\bibitem{Majjaoli}  H. Majjaoli,  {Dunkl-Schr\"odinger semigroups and applications}, \emph{Appl. Anal.} 92(8), 1597-1626 (2013).
		
		\bibitem{shyam} 	S. S. Mondal and J. Swain, Restriction theorem for the Fourier-Hermite transform and solution of the Hermite-Schr\"odinger equation, \emph{Adv. Oper. Theory} 7(4), Paper No. 44 (2022).
		
		\bibitem{shyam1} S. S. Mondal and J. Swain, Strichartz inequality for orthonormal functions associated with special Hermite operator,   arXiv:2103.12586 (2022).
		
		\bibitem{nakamura}	S. Nakamura, The orthonormal  Strichrtz inequality on Torus, \emph{Trans. Amer. Math. Soc.} 373, 1455-1476 (2020).
		
		\bibitem{ros}R. R\"{o}sler, Positivity of Dunkl's intertwining operator, \emph{Duke Math. J.} 98, 445-463 (1999). 
		
		\bibitem{Simon}\label{Simon} B. Simon, Trace ideals and their applications, Cambridge University Press., Cambridge (1979).
		
		\bibitem{stein}  E. M. Stein,   {Interpolation of linear operators},   \emph{Trans. Amer. Math. Soc.}  83,   482-492 (1956).
		
		\bibitem{Stein} E. M. Stein,  \textit{Oscillatory integrals in Fourier analysis}.  In: Beijing Lectures in Harmonic Analysis (Beijing, 1984),  Annals of Mathematics Studies, vol. 112,  Princeton University Press, Princeton  (1986).
		
		
		
		
		
		\bibitem{st} R.  S. Strichartz,    Restrictions of Fourier transforms to quadratic surfaces and decay of solutions of wave equations,  \emph{Duke Math. J.} 44,   705-714  (1977).
		
		
		
		
		\bibitem{tao} T.  Tao,  \textit{Some recent progress on the restriction conjecture}.  In: Fourier analysis and convexity, pp. 217-243, Birkh\"auser, Boston (2004).
		
		\bibitem{T1} P. A. Tomas,   {A restriction theorem for the Fourier transform},  \emph{Bull. Amer. Math. Soc.} 81, 477-478 (1975).
		
		\bibitem{T2} P. A. Tomas,   \emph{Restriction theorems for the Fourier transform},   In: Proc. Symp. Pure Math.,  XXXV (1979).
		
		
		
		
		
		
		
		
		
		
		
	\end{thebibliography}
\end{document}